\documentclass[11pt]{article}
\usepackage{amsmath, amssymb, bm, graphicx, nicefrac,a4wide, amsthm}
\usepackage{psfrag}
\usepackage[usenames,dvipsnames]{color}
\usepackage{xcolor}
\usepackage[all,cmtip,line]{xy}
\usepackage[normalem]{ulem}
\usepackage{tikz,tikz-cd}
\usetikzlibrary{cd}
\usetikzlibrary{matrix, arrows}
\usepackage{pgfplots}
\usepackage{subfig}
\usepackage{arcs}
\usepackage[english]{babel}
\usepackage[utf8]{inputenc}
\usepackage{authblk}

\usepackage[T3,T1]{fontenc}
\DeclareSymbolFont{tipa}{T3}{cmr}{m}{n}
\DeclareMathAccent{\invbreve}{\mathalpha}{tipa}{16}

\newtheorem{theorem}{Theorem}[section]

\newtheorem{lemma}[theorem]{Lemma}

\newtheorem{definition}[theorem]{Definition}
\newtheorem{assumption}[theorem]{Assumption}

\newtheorem{problem}[theorem]{Problem}
\newtheorem{remark}[theorem]{Remark}

\newcommand{\seminorm}[2]{\left\vert #1 \right\vert_{#2}}
\newcommand{\RR}{\mathbb{R}}
\newcommand{\curl}{\operatorname{\mathbf{curl}}}
\renewcommand{\div}{\operatorname{div}}
\newcommand{\p}[1]{\langle #1\rangle}
\newcommand{\pr}[1]{\left( #1\right)}

\newcommand{\modulo}[1]{\left\vert #1\right\vert}

\newcommand{\hcurl}[1]{\bm{H}( \curl; #1 )}
\newcommand{\hcurlcurl}[2]{\bm{H}_{#2}( \curl\curl; #1 )}
\newcommand{\hscurl}[2]{\bm{H}^{#2}( \curl; #1 )}
\newcommand{\hocurl}[1]{\bm{H}_0(\curl; #1 )}
\newcommand{\hdiv}[1]{\bm{H}(\div;#1)}
\newcommand{\hodiv}[1]{\bm{H}_0(\div;#1)}
\newcommand{\Hsob}[2]{\bm{H}^{#1}( #2 )}
\newcommand{\hsob}[2]{{H}^{#1}( #2 )}

\newcommand{\hosob}[2]{{H}_0^{#1}( #2 )}
\newcommand{\Lp}[2]{\bm{L}^{#1}( #2 )}
\newcommand{\lp}[2]{{L}^{#1}( #2 )}

\newcommand{\be}{\begin{equation}}
\newcommand{\ee}{\end{equation}}

\newcommand{\eps}{{\varepsilon}}

\newcommand{\C}{\mathcal C}

\newcommand{\cE}{\mathcal E}

\newcommand{\cO}{\mathcal O}

\newcommand{\cJ}{\mathcal J}

\newcommand{\dist}{\mathrm{dist}}

\newcommand{\bnul}{{\boldsymbol 0}}

\newcommand{\bM} {\mathbf{M}}

\newcommand{\bF} {\mathbf{F}}

\newcommand{\D}{\mathrm{D}}

\newcommand{\dup}[2]{\langle #1, #2\rangle}

\renewcommand{\bar}[1]{\overline#1}
\renewcommand{\div}{{\rm div}}

\newcommand{\IC}{{\mathbb C}}

\newcommand{\IN}{{\mathbb N}}

\newcommand{\IR}{{\mathbb R}}

\newcommand{\half}{\frac{1}{2}}

\newcommand{\bp}{{\bm p}}

\newcommand{\norm}[2]{\left\lVert #1\right\rVert_{#2}}
\newcommand{\tnorm}[1]{{\left\vert\kern-0.25ex\left\vert\kern-0.25ex\left\vert #1 
    \right\vert\kern-0.25ex\right\vert\kern-0.25ex\right\vert}}

\newcommand{\rK}{{\breve{K}}}

\renewcommand{\d}{\operatorname{d}}

\renewcommand{\div}{\operatorname{div}}

\newcommand{\rst}[1]{\left.#1\right|}  

\newcommand{\scurl}{\operatorname{curl}}

\newcommand{\bE}{\mathbf{E}}
\newcommand{\bH}{\boldsymbol{H}}
\newcommand{\VH}{\mathbf{H}}

\newcommand{\bn}{\mathbf{n}}

\newcommand{\bx}{\mathbf{x}}
\newcommand{\by}{\mathbf{y}}
\newcommand{\bJ}{\mathbf{J}}

\newcommand{\bW}{\boldsymbol{W}}

\newcommand{\bV}{\mathbf{V}}
\newcommand{\bU}{\mathbf{U}}

\newcommand{\href}{\VH^{\mathrm{ref}}}

\newcommand{\tD}{\gamma_{\mathrm{D}}}
\newcommand{\tN}{\gamma_{\mathrm{N}}}

\numberwithin{equation}{section}

\begin{document}
\title{The effect of quadrature rules on finite element solutions
of Maxwell variational problems. Consistency estimates on meshes with straight and curved elements.\thanks{This work was supported in part by Fondecyt Regular 1171491
and doctoral grant Conicyt-PFCHA 2017-21171791.}}

\author[1]{Rub\'en Aylwin}
\author[2]{Carlos Jerez-Hanckes}
{\tiny 
\affil[1]{Department of Electrical Engineering, Pontificia Universidad Cat\'olica de Chile, Santiago, Chile.}
\affil[2]{Faculty of Engineering and Sciences, Universidad Adolfo Ib\'a\~nez, Santiago, Chile.}
}

\maketitle

\begin{abstract}
We study the effects of numerical quadrature rules on error convergence rates when solving Maxwell-type variational problems via the curl-conforming or edge finite element method. A complete {\em a priori} error analysis for the case of bounded polygonal and curved domains with non-homogeneous coefficients is provided.
We detail sufficient conditions with respect to mesh refinement and precision for the quadrature rules so as to guarantee convergence rates following that of exact numerical integration. On curved domains, we isolate the error contribution to numerical quadrature rules.
\end{abstract}

\section{Introduction}
%
\label{sec:Intro}
We provide a complete error analysis on the effects
of numerical integration when approximating Maxwell
solutions via finite elements (FEs). Specifically, we
consider a range of problems set in bounded domains with
perfectly conducting boundary conditions (PEC). Through
our analysis, we find conditions for quadrature rules to
guarantee orders of convergence with respect to the
mesh-size of the error associated with numerical
approximation of exact integration.

Strang-type lemmas for have been long available for different types of problems: elliptic 
\cite{banerjee1992effect,ciarlet1991basic,ciarlet1972combined,Ciarlet:2002aa};
non-linear elliptic  \cite{abdulle2012priori}; fourth-order elliptic
problems \cite{bhattacharyya1999combined}; and, eigenvalue problems
\cite{banerjee1992note,banerjee1989estimation,vanmaele1995combined}. However, and to our knowledge, similar results
for Maxwell-type variational formulations are unavailable, thus motivating the present work. 

Our main results are Theorems \ref{thm:mainresaffine} and
\ref{thm:mainrescurv}, in Sections \ref{sec:PolyDom} and
\ref{sec:CurvDom}, respectively. The latter presents an
estimate for the error convergence rate between fully discrete
and continuous solutions on polygonal domains, specifying
sufficient conditions on quadrature rules to ensure the same
convergence rate one would obtain with exact integration.
The former drops the assumption that the domain be polygonal
and gives conditions on quadrature rules to ensure a desired
convergence rate of the error terms that spawn from considering
numerical integration. As is to be expected, conditions on quadrature rules will depend on the polynomial degree of
FE approximation spaces and the degree of precision used
to mesh the domain when not polygonal. Smoothness of parameters and of the continuous
solution will only limit the maximum possible
convergence rate.

In \cite{AJZS18}, we showed error estimates for fully discrete solutions
of a Maxwell-type problem with inhomogeneous coefficients on a
tetrahedral and quasi-uniform sequence of affine
meshes, which was of importance in the uncertainty quantification
(UQ) setting there considered. Our present results can be regarded
as a generalization of Theorem 3.19 in \cite{AJZS18} to account for
inhomogeneous and/or anisotropic materials as well as for the implementation
of meshes with curved elements (\emph{cf}.~\cite[Sec.~8.3]{Monk:2003aa}
and references therein). As in \cite{AJZS18}, a key tool throughout our analysis
on affine meshes---with straight tetrahedrons as elements---will be the
quasi-interpolation operators developed in \cite{Ern:2017ab}. These operators
require very low smoothness: no greater than $L^1$ from the interpolated
function, whereas the canonical interpolation operator requires a minimum
smoothness (\emph{cf.}~\cite{ern2004theory,Ern:2017ab,Monk:2003aa}).
Coupling these results with standard estimates for the error convergence of FE
solutions allow us to present a complete analysis of the convergence
of fully discrete solutions of Maxwell equations on polyhedral domains.

We shall not consider an analogous result on curved meshes as their use
proves advantageous only when the solution has some minimum smoothness.
As interpolation on curved elements lies beyond the scope of this work,
we refer to \cite{bhattacharyya1999combined,ciarlet1972combined,hernandez2003finite,lenoir1986}
and references therein as examples of strategies when dealing with curved boundaries. Incidentally,
we will mainly use the strategy presented in \cite{ciarlet1972combined}
to estimate the perturbations generated by the introduction of quadrature
rules. Hence, we shall isolate the impact of numerical integration
on the error convergence rate of fully discrete solutions and seek to find
convergence rates for those specific contributing terms.

The structure of the article is as follows. In Section 
\ref{sec:ErrAnalysis} we set notation to be used
throughout, introduce Maxwell equations and fix
the general structure of the variational problems considered.
Sections \ref{sec:PolyDom} and \ref{sec:CurvDom} concern
themselves with the analysis of the convergence rates of
fully discrete solutions,~i.e.,~when considering numerical integration
of the previously introduced problems on polygonal domains
and on domains with curved boundaries, respectively. Numerical examples are presented in Section \ref{sec:Num_Ex} and are followed by
concluding remarks in Section \ref{sec:Conc}.

\section{General definitions and Maxwell variational problems}
\label{sec:ErrAnalysis}
We start by setting the notation used in the following sections,
and continue by introducing the general form of the variational problems analysed.
\subsection{Notation}
For $d=1,2,3$ we consider $\Omega\subset\IR^{d}$ an open and
bounded Lipschitz domain. For $m\in\IN$, $\C^m(\Omega)$
denotes the set of real valued functions with $m$-continuous
derivatives on $\Omega$.
For $k$ and $q\in\IN$, $\mathbb{P}_k(\Omega;\IC^q)$
denotes the space of functions from $\Omega$ to $\IC^q$ with
polynomials of degree less than or equal to $k$ in their $q$ components and
$\widetilde{\mathbb{P}}_{k}(\Omega;\IC^q)$ denotes the space of elements of
$\mathbb{P}_k(\Omega;\IC^q)$ of degree exactly $k$ in their $q$ components.

Let $p\geq 1$ and $s\in\IR$, then
$L^p(\Omega)$  and $W^{p,s}(\Omega)$ denote the class
of $p$-integrable functions on $\Omega$ with values in $\IC$ and 
the standard Sobolev spaces, respectively. If $p=2$, we use the
standard notation $H^s(\Omega):=W^{2,s}(\Omega)$.
Boldface symbols will be used to differentiate general
scalar valued function spaces from their vector valued counterparts.

Norms and seminorms over a general Banach space $Y$ are
indicated by subscript ($\norm{\cdot}{Y}$ and $\seminorm{\cdot}{Y}$). We
make an exception for $H^s(\Omega)$, whose norm and seminorm will be written
as $\norm{\cdot}{s,\Omega}$ and $\seminorm{\cdot}{s,\Omega}$. The dual of
the Banach space $Y$ is denoted $Y'$.

For a Hilbert space $X$ (real or complex) its inner product is denoted as
$\pr{\cdot,\cdot}_X$, while duality products are denoted by
$\p{\cdot,\cdot}_{X'\times X}$. The same exception is made for
Sobolev spaces as in the case of norms and seminorms. Both
duality and inner products are understood in the sesqulinear
sense.
\subsection{Functional spaces}
\label{sec:funspaces}
We shall require the following functional space of vector valued
functions with integrable curl and divergence:
\begin{gather*}
\hcurl{\Omega}:=\left\{\bU\in \Lp{2}{\Omega} \ :\ \curl\bU\in \Lp{2}{\Omega} \right\},\\
\hdiv{\Omega}:=\left\{\bU\in \Lp{2}{\Omega} \ :\ \div\bU\in \Lp{2}{\Omega} \right\},
\end{gather*}
which are Hilbert spaces when paired with the following inner products
\begin{gather*}
\pr{\bU,\bV}_{\hcurl{\Omega}}:=\pr{\bU,\bV}_{{0},{\Omega}}+\pr{\curl\bU,\curl\bV}_{{0},{\Omega}},\\
\pr{\bU,\bV}_{\hdiv{\Omega}}:=\pr{\bU,\bV}_{{0},{\Omega}}+\pr{\div\bU,\div\bV}_{{0},{\Omega}}.
\end{gather*}
For $s>0$, we introduce the following extension to $\hcurl{\Omega}$
of functions in $\Hsob{s}{\Omega}$ with curl in $\Hsob{s}{\Omega}$
\cite[Sec.~3.5.3]{Monk:2003aa}:
\begin{align*}
\hscurl{\Omega}{s}:=\left\{\bU\in \Hsob{s}{\Omega} \ :\ \curl\bU\in \Hsob{s}{\Omega} \right\}
\end{align*}
with norm
\begin{align*}
\norm{\bU}{\bH^s(\curl;\Omega)}:=
\left(\norm{\curl\bU}{s,\Omega}^2+\norm{\bU}{s,\Omega}^2\right)^{\frac{1}{2}}.
\end{align*}
We also introduce the following subspace of $\hcurl{\Omega}$,
\begin{align*}
\hcurlcurl{\Omega}{}:=\left\lbrace \bU\in\hcurl{\Omega}\ :\ 
\curl\curl\bU\in\Lp{2}{\Omega} \right\rbrace
\end{align*}
and following trace spaces \cite{BuCo00,BufHipTvPCS_NM2003,Monk:2003aa}:
\begin{align*}
\bH^{-\half}_{\div}(\partial\Omega) := \{\bU\in
\bH^{-\half}(\partial\Omega):\bU\cdot\bn = 0,\; \div_{\partial\Omega}\bU \in
H^{-\half}(\partial\Omega) \},\\
\bH^{-\half}_{\scurl}(\partial\Omega) := \{\bU\in
\bH^{-\half}(\partial\Omega):\bU\cdot\bn = 0,\; \scurl_{\partial\Omega}\bU
\in H^{-\half}(\partial\Omega) \}
\end{align*}
where $\bn$ is the outward normal vector
from $\Omega$, $\div_{\partial\Omega}$ is the surface divergence
operator and $\scurl_{\partial\Omega}$ is the surface scalar curl
operator, respectively (\emph{cf.}~\cite{BuCo00,BufHipTvPCS_NM2003}).
Also, by \cite[Thm.~2]{BuHip} note that
$$\bH^{-\half}_{\scurl}(\partial\Omega)=
\left(\bH^{-\half}_{\div}(\partial\Omega)\right)'.$$ 
\begin{definition} \label{def:traces} Let
$\bU\in \boldsymbol{\C}^\infty(\bar{\Omega})$, then
\begin{gather*}
\tD\bU := \rst{\bn \times(\bU\times \bn)}_{{\partial \Omega}},
\quad\tD^\times\bU := \rst{(\bn\times\bU)}_{{\partial \Omega}},\\
\gamma_{\bn}\bU := \rst{(\bn\cdot\bU)}_{{\partial \Omega}},
\quad\mbox{and}\quad \tN\bU := \rst{(\bn\times
\curl\bU)}_{{\partial \Omega}} \;
\end{gather*}
are the Dirichlet trace, flipped Dirichlet trace,
normal trace and Neumann trace, respectively.
\end{definition}

The Dirichlet trace operators in Definition \ref{def:traces}
can be extended to linear and continuous operators
from $\bH(\curl;\Omega)$ to  $\bH^{-\half}({\partial\Omega})$
\cite[Thms.~3.29 and 3.31]{Monk:2003aa}. Specifically,
one sees that
\begin{align*}
\mathrm{Im}(\tD)=\bH^{-\half}_{\scurl}({\partial \Omega}),\qquad
\mathrm{Im}(\tD^{\times})=\bH^{-\half}_{\div}(\partial\Omega)
\end{align*}
allowing us to endow this spaces with the corresponding
graph norms given by the trace operators. Similarly, the
normal trace operator can be extended to a linear and
continuous operator \cite[Thm.~3.24]{Monk:2003aa}:
\begin{align*}
\gamma_{\bn}:\hdiv{\Omega}\rightarrow \bH^{-\frac{1}{2}}({\partial \Omega})
\end{align*}
while the Neumann trace may be extended as
\cite[Thm.~3.2]{BuCo00}
\begin{align*}
\gamma_{N}:\hcurlcurl{\Omega}{}\rightarrow \bH^{-\half}_{\div}(\partial\Omega).
\end{align*}

With the trace operators $\tD^\times$ and $\gamma_{\bn}$, we define 
\begin{gather}
\label{eq:H0curl}
\bH_0(\curl;\Omega) := \{\bU\in \bH(\curl;\Omega) \ : \ \tD^\times \bU =
\bnul \; \mbox{on} \; \partial \Omega \}\;,  \\
\label{eq:H0div}
\bH_0(\div;\Omega) := \{\bU\in \bH(\div;\Omega) \ : \ \gamma_{\bn} \bU =
\bnul \; \mbox{on} \; \partial \Omega \}.
\end{gather}
By continuity of $\tD^\times$, $\bH_0(\curl;\Omega)$ is a closed
subspace of $\bH(\curl;\Omega)$ (analogously,
$\bH_0(\div;\Omega)$ is a closed
subspace of $\bH(\div;\Omega)$). Finally, for $\bU$ and
$\bV\in \bH(\curl,\Omega)$ there holds \cite[Eq.~(27)]{BuCo00}:
\begin{equation}\label{eq:green}
(\bU,\curl\bV)_\Omega-(\curl\bU,\bV)_\Omega =
-\dup{\tD^\times\bU}{\tD\bV}_{\partial\Omega}\;
\end{equation}
where $\dup{\cdot}{\cdot}_{\partial\Omega}$ denotes the
duality between $\bH_{\mathrm{div}}^{-\half}(\partial \Omega)$
and $\bH^{-\half}_{\scurl}(\partial\Omega)$.

\subsection{Maxwell Equations}
We consider an open bounded Lipschitz domain
$\D\subset\IR^3$ with boundary $\Gamma=\partial\D$
as well as a time-harmonic dependence $e^{\imath \omega t}$
with circular frequency $\omega > 0$. We write $\bE$ and
$\VH$ for the complex-valued electric and magnetic fields,
respectively. Harmonic Maxwell equations on $\D$ read
\begin{align}\label{eq:Maxwell}
\begin{aligned}
\curl \bE + \imath\omega \mu \VH &= \bnul,\\
\imath\omega\eps  \bE -\curl \VH &= \bJ,
\end{aligned}
\end{align}
where $\mu$ and $\eps$ are assumed to be symmetric
matrix-valued functions with coefficients in $\lp{\infty}{\D}$,
and $\bJ$ is an imposed current, usually---but not necessarily---compactly supported in $\D$.
\begin{assumption}[Basic assumptions on the parameters]
Both $\mu$ and $\epsilon$ are symmetric complex-matrix
valued functions with coefficients in $\lp{\infty}{\D}$.
Furthermore, $\mu$ has a pointwise inverse, denoted
$\mu^{-1}$, almost everywhere on $\D$.
\end{assumption}

The Maxwell system \eqref{eq:Maxwell} is commonly reduced
to a second order partial differential equation by removing
either $\bE$ or $\VH$. We consider the following reduction:
\begin{align}\label{eq:MaxE}
\curl \mu^{-1}\curl \bE - \omega^2\eps \bE = -\imath\omega\bJ.
\end{align}
The system is completed by imposing PEC boundary
conditions on the surface $\Gamma$
\begin{align}\label{eq:PECCond}
\tD^{\times}\bE=0.
\end{align}
\subsection{Variational formulation}
We proceed as in \cite{AJZS18} and introduce the
sesquilinear and antilinear forms associated to equations
\eqref{eq:MaxE} and \eqref{eq:PECCond}, defined for
$\bU$ and $\bV\in\bH_0(\curl;\D)$
\begin{align}
\label{eq:Phi}
\Phi(\bU,\bV)&:=
\int_{\D} \mu^{-1} \curl \mathbf{U} \cdot \curl \overline{\mathbf{V}}-
\omega^2\epsilon \mathbf{U} \cdot \overline{\mathbf{V}} \d\!\bx,\\
\label{eq:rhs}
\bF(\bV)&:=-\imath\omega\int_{\D}\bJ\cdot\overline{\bV}\d\!\bx,
\end{align}
both continuous on $\bH_0(\curl;\D)$.

\begin{problem}[Continuous variational problem]
\label{prob:varprob}%
Find $\bE\in \bH_0(\curl;\D)$ such that
\begin{align*}
\Phi(\bE,\bV)=\bF(\bV),
\end{align*}
for all $\bV\in \bH_0(\curl;\D)$.
\end{problem}
Since we are only interested in the effect of
numerical integration when discretizing Problem
\ref{prob:varprob}, we assume the sesquilinear form
in \ref{eq:Phi} satisfies all necessary conditions for there
to be a unique solution of Problem \ref{prob:varprob} that
depends continually on the data.

\begin{assumption}[Wellposedness]
\label{ass:sesqform}
We assume the sesquilinear form $\Phi$ in \eqref{eq:Phi}
to satisfy the following conditions:
\begin{gather*}
\sup_{\bU\in \hocurl{\D}\setminus\{\bnul\}}\modulo{\Phi(\bU,\bV)}>0\quad
\forall\;\bV\in \hocurl{\D}\setminus\{\bnul\},\\
\inf_{\bU\in \hocurl{\D}\setminus\{\bnul\}}\left(\sup_{\bV\in \hocurl{\D}\setminus\{\bnul\}} 
\frac{\modulo{\Phi(\bU,\bV)}}{\norm{\bU}{\hcurl{\D}}\norm{\bV}{\hcurl{\D}}}\right)\geq C>0.
\end{gather*}
\end{assumption}

By Assumption \ref{ass:sesqform} and the continuity of
$\Phi$ and $\bF$ in \eqref{eq:Phi} and \eqref{eq:rhs},
there exists a unique solution $\bE\in\hocurl{\D}$
for Problem \ref{prob:varprob}. For examples of
variational problems with a structure analogous to that of 
$\Phi$ in \eqref{eq:Phi} we refer to \cite{ern2018analysis},
where two different problems concerning Maxwell equations
are found to be coercive---i.e.~$\modulo{\Phi(\bU,\bU)}/\norm{\bU}{\hcurl{\D}}^2\geq
\alpha>0$ for all $\bU\in\hocurl{\D}$---\cite[Chap.~4.7]{monk03} and references therein (incidentally, the problem
analysed in \cite{AJZS18} in the context of UQ is one of the problems in
\cite{ern2018analysis}).

\section{Finite Elements and Consistency Error Estimates for Polyhedral Domains}
\label{sec:PolyDom}
In what follows, we concern ourselves with discretizations of
Problem \ref{prob:varprob}. We shall construct a sequence
of meshes $\{\tau_h\}_{h>0}$, from which we construct
discrete subspaces of $\hocurl{\D}$ in order to
approximate the solution of Problem \ref{prob:varprob}.
We begin our analysis by assuming  $\D$ to be polyhedral,
so that meshes $\tau_h$ constructed from tetrahedrons cover
$\D$ exactly. We shall extend our analysis to curved domains
and consider non-affine meshes on the following section.

\begin{assumption}[Polyhedral domain]
\label{ass:polydom}
The open domain $\Omega$ is polyhedral.
\end{assumption} 
\subsection{Finite elements}
\label{ssec:FEMPoly}
Let $\{\tau_{h_i}\}_{i\in\IN}$ be a sequence of
quasi-uniform meshes constructed from disjoint, matching
tetrahedrons---$K\in\tau_h$ for each mesh $\tau_h$ in the sequence---that cover $\Omega$ exactly, where the subindex $h>0$
refers to the mesh-size of each mesh in the sequence and where
$h_i\rightarrow 0$ as $i\in\IN$ grows to infinity. 
\begin{assumption}[Assumptions on the sequence of meshes]
\label{ass:meshPoly}
 The meshes in the sequence $\{\tau_{h_i}\}_{i\in\IN}$ are affine,
 quasi-uniform and cover $\Omega$
 exactly.
\end{assumption}
\begin{definition}[Reference element]
We define $\breve{K}$ as the tetrahedron with vertices
$\bnul$, $\bm{e}_1$, $\bm{e}_2$ and $\bm{e}_3$; and
refer to it as the reference element.
\end{definition}
\begin{definition}\label{def:affinebijectivemap}
For any $i\in\IN$ and each $K\in\tau_{h_i}$ we define
$T_K:\breve{K}\mapsto K$ as affine, bijective mappings
from the reference tetrahedron to arbitrary $K\in\tau_{h_i}$.
We denote the Jacobians of these mappings as $\mathbb{J}_K$.
\end{definition}
The elements from the mesh, i.e.~$K\in\tau_h$,
may be considered as constructed from the reference
tetrahedron $\breve{K}$ through the mappings
$T_K$ introduced in Definition \ref{def:affinebijectivemap}.

\begin{definition}[Finite elements]\label{def:femtriple}
We will consider finite elements as triples $(K,P_K,\Sigma_K)$,
with $K\in\tau_h$, $P_K$ a space of functions over $K$ (usually polynomials) and
$\Sigma_K:=\{\sigma^{K}_{i}\}_{i=1}^{n_\Sigma},\ n_\Sigma\in\IN$ a set of linear
functionals acting on $P_K$ (\emph{cf.}~\cite{Monk:2003aa}).
\end{definition}
Let $k\in\IN$. Since we are considering only Maxwell equations,
we will only work with the finite element
$({K},{\bm{P}}^c_K,{\Sigma}^c_K)$ as defined in
\cite[Chapter 5]{Monk:2003aa} and corresponding
to curl-conforming elements,
\begin{align*}
{\bm{P}}^c_K&:=\mathbb{P}_{k-1}(K;\RR^3)\oplus \lbrace\bm{p}\in 
\widetilde{\mathbb{P}}_{k}(K,\RR^3)\ : \bx\cdot\bp=0\rbrace.
\end{align*}
For completeness, we also introduce the function spaces for grad-,
div and $L^2$-conforming finite elements:
\begin{align*}
{P}^g_K:=\mathbb{P}_{k}({K};\IC),\quad
{\bm{P}}^d_K:=\mathbb{P}_{k-1}(K;\IC^3)\oplus\lbrace\bx p\ :\
p\in\widetilde{\mathbb{P}}_{k-1}(K,\IC)\rbrace,\quad
{P}^b_K:=\mathbb{P}_{k-1}({K};\IC).
\end{align*}
We refer to \cite[Chapter 5]{Monk:2003aa} for
the definition of the degrees of freedom $\Sigma^g_K$,
$\Sigma^c_K$ and $\Sigma^d_K$, corresponding to the spaces
$P^g_K$, $\bm{P}^c_K$ and $\bm{P}^d_K$ respectively.

From here onwards, let $k\in\IN$ be fixed as the polynomial
degree of our approximation spaces and let $\tau_h$ be an
arbitrary mesh in the sequence $\{\tau_{h_i}\}_{i\in\IN}$, where
the subindex $h$ represents the size of the mesh, as before.
Discrete spaces on $\tau_h$ are constructed as follows 
\begin{align*}
P^g(\tau_h):= &\left\lbrace v_h\in\hsob{1}{\Omega} \ :\ 
v_h\vert_K\in P^g_K\right\rbrace,\\
\bm P^c(\tau_h):= &\left\lbrace \bV_h\in\hcurl{\Omega} \ :\  
\bV_h\vert_K \in {\bm{P}}^c_K\quad \forall\; K\in\tau_h\right\rbrace,\\
\bm P^d(\tau_h):= &\left\lbrace \bV_h\in\hdiv{\Omega} \ :\  
\bV_h\vert_K \in {\bm{P}}^d_K\quad \forall\; K\in\tau_h\right\rbrace,\\
P^b(\tau_h):= &\left\lbrace v_h\in\lp{2}{\Omega} \ :\ 
v_h\vert_K\in {P}^b_K\quad\forall\; K\in\tau_h  \right\rbrace.
\end{align*}
Homogeneous essential boundary conditions
are accounted for by imposing the conditions at the boundary:
\begin{align*}
P^g_0(\tau_h)&:= P^g(\tau_h)\cap\hosob{1}{\Omega},\\
\bm P^c_0(\tau_h)&:=\bm P^c(\tau_h)\cap\hocurl{\Omega},\\
\bm P^d_0(\tau_h)&:=\bm P^d(\tau_h)\cap\hodiv{\Omega}.
\end{align*}

We introduce, for every $K\in\tau_h$
the following pullbacks to the reference element $\breve{K}$:
\begin{gather}
\begin{gathered}
\psi^g_K(v):=v\circ T_K,\quad
\psi^c_K(\bV):=\mathbb{J}_K^\top (\bV\circ T_K),\\
\psi^d_K(\bV):=\det (\mathbb{J}_K)\mathbb{J}_K^{-1}(\bV\circ T_K),\quad
\psi^b_K(v):=\det (\mathbb{J}_K)(v\circ T_K),
\end{gathered}\label{eq:pullbacks}
\end{gather}
where $v\in \lp{p}{K}$ and $\bV\in\Lp{p}{K}$, $p\geq 1$. 
We continue by stating some useful properties of the
pullbacks in \eqref{eq:pullbacks} and refer to
\cite{ern2004theory} for their proofs.
First, the mappings in \eqref{eq:pullbacks} commute
with the differential operators, i.e.
\begin{align*}
\nabla\psi^g_K(v)=\psi^c_K(\nabla v),\quad \curl\psi^c_K(\bV)=
\psi^d_K(\curl \bV),\quad \div\psi^d_K(\bV)=\psi^b_K(\div \bV),
\end{align*} 
for all $K\in\tau_h$, for all functions $v$ with well defined
gradient, and $\bV$ with well defined curl or divergence,
respectively. Furthermore, the finite element spaces for
functions with well defined gradient, curl and divergence
are invariant with respect to their respective pullback to
$\breve{K}$ so that, for all $K\in\tau_h$ and $j\in\{ g,c,d,b\}$,
it holds
\begin{align*}
\psi_{K}^{j}:P^{j}_K\rightarrow P^{j}_{\breve{K}},\qquad 
(\psi_{K}^{j})^{-1}:P^{j}_{\breve{K}}\rightarrow P^{j}_K.
\end{align*}
Under Assumption \ref{ass:meshPoly}, there exist uniform
positive constants $c^\sharp$ and $c^\flat$ such that:
\begin{gather}\label{eq:JacboundK}
\modulo{\det(\mathbb{J}_K)}=\modulo{K}\modulo{\breve{K}}^{-1},\qquad
\norm{\mathbb{J}_K}{\RR^{3\times 3}}\leq c^{\sharp}h,\qquad
\norm{\mathbb{J}_K^{-1}}{\RR^{3\times 3}}\leq c^{\flat}h^{-1}.
\end{gather}
As in \cite{Ern:2017ab,ern2018analysis}, we continue by summarizing the
linear mappings defined in \eqref{eq:pullbacks} as
\begin{align}
\psi_K(v) = \mathbb{A}_K(v\circ T_K),\quad\mbox{or }
\quad \psi_K(\bV) = \mathbb{A}_K(\bV\circ T_K),\label{eq:psisummary}
\end{align}
to avoid repeating the properties \eqref{eq:struct1K} and \eqref{eq:struct2K},
where $\mathbb{A}_K=1$, $\mathbb{J}^{\top}_{K}$,
$\det(\mathbb{J}_K)\mathbb{J}^{-1}_K$ or
$\det(\mathbb{J}_K)$ and $q=1$ or $3$
depending on the choice of $\mathbb{A}_K$.
Then, for all $K\in\tau_h$, $p\geq 1$ and $l\in\IN_0$ the mappings in
\eqref{eq:pullbacks} satisfy 
\begin{gather}
\seminorm{\psi_K}{\mathcal{L}(W^{l,p}(K);W^{l,p}(\breve{K}))}\leq
c\norm{\mathbb{A}_K}{\RR^{q\times q}}\norm{\mathbb{J}_K}{\RR^{3\times 3}}^{l}
\modulo{\det(\mathbb{J}_K)}^{-\frac{1}{p}},\label{eq:struct1K}\\
\seminorm{\psi_K^{-1}}{\mathcal{L}(W^{l,p}(\breve{K});W^{l,p}({K}))}\leq
c\norm{\mathbb{A}_K^{-1}}{\RR^{q\times q}}\norm{\mathbb{J}_K^{-1}}{\RR^{3\times 3}}^{l}
\modulo{\det(\mathbb{J}_K)}^{\frac{1}{p}}.\label{eq:struct2K}
\end{gather}
We continue by stating some Lemmas that will be required
during the proof of the main result of this section.
   \begin{lemma}[Lemma 1.138 in \cite{ern2004theory}]
   \label{lemma:inverseineq}
    Let $K\in\tau_h$ and $m$, $l\in\IN_0$ such that $m\leq l$. Then,
    under Assumption \ref{ass:meshPoly}, for $\phi$ in either $P^g_K$,
    $\bm{P}^c_K$, $\bm{P}^d_K$ or ${P}^b_K$,
    \begin{align*}
      \norm{\phi}{W^{l,p}(K;\IR^q)}\leq ch^{m-l}\norm{\phi}{W^{m,p}(K;\IR^q)},
    \end{align*}
    for a positive constant $c$ independent of $K$ and the mesh $\tau_h$ and $q=1$
    or $3$ depending on the finite element.
 \end{lemma}
 
 \begin{lemma}[Quasi-interpolation operator \cite{Ern:2017ab}]
\label{lemma:quasiInt}
Let $K\in\tau_h$. For each one of the function spaces $\bm P^c(K)$ and
$\bm P^d(K)$ there exists a quasi-interpolation operator, denoted
$\mathcal{I}_K^{\#,c}$ and $\mathcal{I}_K^{\#,d}$
respectively, and a positive constant $c>0$ independent of
$K$ and $h$ such that for all $\bV\in\bm W^{m,p}(K)$,
$\mathcal{I}_K^{\#,c}(\bV)\in\bm P^c(K)$, $\mathcal{I}_K^{\#,d}(\bV)\in\bm P^d(K)$,
\begin{align*}
\seminorm{\bV-\mathcal{I}_K^{\#,c}(\bV)}{\bm W^{m,p}(K)}\leq ch_{K}^{r-m}\seminorm{\bV}{\bm W^{r,p}(K)},\\
\seminorm{\bV-\mathcal{I}_K^{\#,d}(\bV)}{\bm W^{m,p}(K)}\leq ch_{K}^{r-m}\seminorm{\bV}{\bm W^{r,p}(K)},
\end{align*}
for all $m\in\IN$, $r\in\IR$ and $p\in\IR$ with $p\geq1$ and $m\leq r\leq k$,
where $h_K$ is the diameter of $K$. Furthermore, $\bm P^c(K)$ and
$\bm P^d(K)$ are invariant with respect to their corresponding
quasi-interpolation operators.
\end{lemma}

\begin{lemma}[Lemma 3.5 in \cite{AJZS18}]
\label{lemma:IBound}
Let $K\in \tau_h$ and $\mathcal{I}^{\#}_K$ denote either
$\mathcal{I}^{\#,c}_K$ or $\mathcal{I}^{\#,d}_K$. Then there
exists a constant $c>0$
independent of $K$ and $h$ such that
\begin{align}
\norm{\mathcal{I}^{\#}_K(\bV)}{\bm{W}^{m,p}(K)}&
\leq c\norm{\bV}{\bm{W}^{m,p}(K)},\label{eq:ilpstab}\\
\norm{\mathcal{I}^{\#}_K(\bV)}{\bm{W}^{m,p}(K)}&
\leq ch^{-1}\norm{\bV}{\bm{W}^{m-1,p}(K)},\label{eq:ilpstab-1}
\end{align}
for all $m\in\IN$ with $m\leq k$ and
$p\in[1,\infty]$. Furthermore, for $r\in\IR$ such that $r\leq k$
it holds
\begin{align}
\norm{\mathcal{I}^{\#}_K(\bV)}{\Hsob{\lceil r \rceil}{K}}
\leq ch^{r-\lceil r \rceil}\norm{\bV}{\Hsob{r}{K}}.\label{eq:sobinterpstab}
\end{align}
\end{lemma}
\begin{proof}
The estimates in \eqref{eq:ilpstab} and \eqref{eq:ilpstab-1} are
consequence of the error estimate in Lemma \ref{lemma:quasiInt},
while \eqref{eq:sobinterpstab} follows from both \eqref{eq:ilpstab}
and \eqref{eq:ilpstab-1} by an application of real interpolation
between Sobolev spaces (\emph{cf.}~\cite[Lemma 22.3]{tartar2007}).
\end{proof}
\subsection{Discrete variational problem}
With the previous definitions of curl-conforming discrete spaces
at hand, we can now state the discrete version of Problem \ref{prob:varprob}.

\begin{problem}[Discrete variational problem on affine meshes]
\label{prob:varprobdis}%
Find $\bE_h\in \bm P^c_0(\tau_h)$ such that
\begin{align*}
\Phi(\bE_h,\bV_h)=\bF(\bV_h),
\end{align*}
for all $\bV_h \in \bm P^c_0(\tau_h)$.
\end{problem}

As with Assumption \ref{ass:sesqform}, we assume our framework to be such
that a unique discrete solution $\bE_h\in\bm P^c_0(\tau_h)$ exists for all
meshes $\tau_h\in\{\tau_{h_i}\}_{i\in\IN}$.

\begin{assumption}[Wellposedness on $\bm P^c_0(\tau_h)$]
\label{ass:sesqformdis}
We assume the sesquilinear form $\Phi$ in \eqref{eq:Phi} to satisfy the following:
\begin{gather*}
\sup_{\bU_h\in \bm P^c_0(\tau_h)\setminus\{\bnul\}}\modulo{\Phi(\bU_h,\bV_h)}>\alpha>0\quad \forall\;\bV_h\in\bm P^c_0(\tau_h)\setminus\{\bnul\},\\
\inf_{\bU_h\in\bm P^c_0(\tau_h)\setminus\{\bnul\}}\left(\sup_{\bV_h\in \bm P^c_0(\tau_h)\setminus\{\bnul\}} 
\frac{\modulo{\Phi(\bU_h,\bV_h)}}{\norm{\bU_h}{\hcurl{\D}}\norm{\bV_h}{\hcurl{\D}}}\right)\geq C>0,
\end{gather*}
on all meshes $\tau_h\in\{\tau_{h_{i}}\}_{i\in\IN}$.
\end{assumption}

From Assumptions \ref{ass:sesqform} and \ref{ass:sesqformdis}, the continuity
of the sesquilinear and antilinear forms in \eqref{eq:Phi} and \eqref{eq:rhs} and
the fact that $\bm P^c_0(\tau_h)\subset\hocurl{\D}$ we see that both Problem
\ref{prob:varprob} and \ref{prob:varprobdis} have unique solutions
$\bE\in\hocurl{\D}$ and $\bE_h\in\bm P^c_0(\tau_h)$, respectively. Furthermore,
if $\bE_{h_i}\in\bm P^c_0(\tau_{h_i})$ solves Problem \ref{prob:varprobdis} on
$\bm P^c_0(\tau_{h_i})$ then
\begin{align*}
\norm{\bE-\bE_{h_i}}{\hcurl{\D}}\leq
C\inf_{\bU_{h_i}\in\bm P^c_0(\tau_{h_i})}\norm{\bE-\bU_{h_i}}{\hcurl{\D}},
\end{align*}
for some $C>0$ independent of the mesh-size.

\subsection{Numerical integration and main results}
\label{sec:numint}

We now introduce quadrature rules for the
numerical computation of the terms for the linear system
associated with Problem \ref{prob:varprobdis}.
\begin{definition}\label{def:quad}
For $L\in\IN$, we define $Q_\rK$, a quadrature rule over $\rK$, as
a linear functional acting on $\phi\in\C(\rK)$ in the following way:
\begin{align*}
 Q_{\breve{K}}(\phi):=
\sum_{l=1}^{L}\breve{w}_{l}\phi(\breve{\bm{b}}_{l}),
\end{align*}
where $\{\breve{w}_{l}\}_{l=1}^{L}\subset \IR$ is a set of quadrature weights
and $\{\breve{\bm{b}}_{l}\}_{l=1}^{L}\subset \rK$ is a set of quadrature points.
\end{definition}
Quadratures over arbitrary elements $K\in\tau_h$ are built from
those in Definition \ref{def:quad}, for $\phi\in\C(K)$, as follows
\begin{align}\label{eq:QK}
  Q_K(\phi):=\sum_{l=1}^L w_{l,K} \phi(\bm{b}_{l,K})\quad\text{with}\quad
  w_{l,K}:=\modulo{\det{(\mathbb{J}_K)}}\breve{w}_{l}\quad\mbox{and}\quad
  {\bm{b}}_{l,K}:=T_K(\breve{\bm{b}}_l),
\end{align}
for $T_K$ and $\mathbb{J}_K$ as in Definition
\ref{def:affinebijectivemap}.
\begin{definition}[Numeric sesquilinear and antilinear form]
\label{def:numsesqant}
Let $Q_\rK^1$, $Q_\rK^2$ and $Q_\rK^3$ be three distinct quadrature rules
as in Definition \ref{def:quad}. We denote by $\widetilde{\Phi}_h({\cdot},{\cdot})$
and $\widetilde\bF_h(\cdot)$ the perturbed, discrete, sesquilinear and antilinear
forms over fields in $\bm{P}^c(\tau_h)$, where exact integration is
replaced by numerical integration:
\begin{align*}
&\widetilde{\Phi}_h({\bU_h},{\bV_h}):=\sum_{K\in\tau_h}
Q_{K}^{1}(\mu^{-1}\curl\bU_h\cdot\curl\overline{\bV_h})
+Q_{K}^{2}(-\omega^2\epsilon\bU_h\cdot\overline{\bV_h}),\\
&\widetilde{\bF}_h({\bV_h}):=\sum_{K\in\tau_h}Q_K^3(-\imath\omega\bJ\cdot \overline{\bV_{h}}),
\end{align*}
where, for $i=1,2,3$, $Q_K^i$ is built from $Q_\rK^i$ as in \eqref{eq:QK}.
\end{definition}

\begin{problem}[Discrete numerical problem]
\label{prob:varprobnum}
Find $\widetilde{\bE}_h\in\bm P^c_0(\tau_h)$ such that,
\begin{align*}
\widetilde\Phi(\widetilde{\bE}_h,{\bV}_h)=
\widetilde\bF({\bV}_h),
\end{align*}
for all ${\bV}_h\in\bm P^c_0(\tau_h)$.
\end{problem}

Our objective is to obtain estimates for the error convergence
rates of the solution of Problem \ref{prob:varprobnum} with respect
to the solution of Problem \ref{prob:varprob}. As such, Strang's lemma
(\emph{cf}.~\cite[Sect.~4.2.4]{SauterSchwabBEM}) will be key
throughout our analysis on this and the following section.

\begin{lemma}[Strang's Lemma. Theorem 4.2.11 in \cite{SauterSchwabBEM}]
\label{lemma:Strang}
Let $\Phi$ in \eqref{eq:Phi} satisfy Assumptions \ref{ass:sesqform}
and \ref{ass:sesqformdis} and let $\bE$ and $\bE_{h_i}$ be the solutions
of Problems \ref{prob:varprob} and \ref{prob:varprobdis}. If the sequence
of sesquiliear forms $\{\widetilde\Phi_{h_i}\}_{i\in\IN}$ given by Definition
\ref{def:numsesqant} satisfies:
\begin{align*}
\modulo{\Phi(\bU_{h_i},\bV_{h_i})-\widetilde\Phi_{h_i}(\bU_{h_i},\bV_{h_i})}
\leq c h_i^r\norm{\bU_{h_i}}{\hscurl{\D}{r}}\norm{\bV_{h_i}}{\hcurl{\D}},\quad
\forall\bU_{h_i},\bV_{h_i}\in\bm P^c_0(\tau_{h_i}),
\end{align*}
for a fixed and positive constant $c$ independent of the
mesh-size, then there is some $\ell\in\IN$ such that for all
the meshes in the sequence $\{\tau_{h_{i}}\}_{i\in\IN,\ i>\ell}$
there exists a unique solution to Problem \ref{prob:varprobnum},
$\widetilde\bE_{h_i}\in\bm P^c_0(\tau_{h_i})$ and
\begin{alignat*}{2}
&\norm{\bE-\widetilde\bE_{h_i}}{\hcurl{\D}} &&\\
&\leq C_S\Bigg(
\norm{\bE-\bE_{h_i}}{\hcurl{\D}}+&&
\sup_{\bV_{h_i}\in\bm P^c_0(\tau_{h_i})\setminus\{\bnul\}}
\frac{\vert{\Phi(\bE_{h_i},\bV_{h_i})-\widetilde\Phi_{h_i}(\bE_{h_i},\bV_{h_i})}\vert}{\norm{\bV_{h_i}}{\hcurl{\D}}} \\
 & &&\qquad\qquad+\sup_{\bV_{h_i}\in\bm P^c_0(\tau_{h_i})\setminus\{\bnul\}}
\frac{\vert{\bF(\bV_{h_i})-\widetilde\bF_{{h_i}}(\bV_{h_i})}\vert}{\norm{\bV_{h_i}}{\hcurl{\D}}} \Bigg)
\end{alignat*}
for a fixed positive constant $C_S$, independent of the mesh-size.
\end{lemma}

We are now ready to state the main result of this section which will be proven presented later on.

\begin{theorem}[Error estimate in affine meshes. Main result of Section \ref{sec:PolyDom}]
\label{thm:mainresaffine}
Let $\bE$ be the unique solution to Problem \ref{prob:varprob} and
suppose the following of $\bE$ and the data of Problem \ref{prob:varprob}:
\begin{align*}
\bE\in\hscurl{\D}{r},\quad \bJ\in\bW^{\lceil r\rceil,q}(\D),\quad\mbox{and}\quad \epsilon_{i,j},\;(\mu^{-1})_{i,j}\in W^{\lceil r\rceil,\infty}(\D),\quad\forall\;i,\;j\in\{1,2,3\}
\end{align*}
for some positive $r\in\IR$ and $q\in\IR$ such that $$r\leq k,\quad q>2\quad\mbox{and}\quad q\geq\frac{\lceil r\rceil}{3}.$$
Then, if quadrature rules used to build
$\widetilde\Phi$ and $\widetilde\bF$ are such that:
\begin{itemize}
\item $Q^1_{\rK}$ is exact for polynomials
of degree $k+\lceil r\rceil-2$,
\item $Q^2_{\rK}$ is exact for polynomials
of degree $k+\lceil r\rceil-1$ and
\item $Q^3_{\rK}$ is exact for polynomials
of degree $k+\lceil r\rceil-1$,
\end{itemize}
there exists some $\ell\in\IN$
such that for all $i>\ell$ there exists a unique solution
$\widetilde\bE_{h_i}\in\bm P^c_0(\tau_{h_i})$ to Problem
\ref{prob:varprobnum} and the solutions satisfy
\begin{align*}
\norm{\bE-\widetilde\bE_{h_i}}{\hcurl{\D}}\leq C_1h_{i}^r\norm{\bE}{\hscurl{\D}{r}}+C_2h_i^{\lceil r\rceil},
\end{align*}
where the positive constants $C_1$ and $C_2$ are independent of
the mesh-size, but depend on the parameters of Problem \ref{prob:varprob}
($\mu$, $\epsilon$, $\omega$, $\bJ$ and $\D$).
\end{theorem}

\subsection{Consistency error estimates and proof of Theorem \ref{thm:mainresaffine}}
We now find error estimates for the quadrature approximation given
in Definition \ref{def:quad} of integrals defining the sesquilinear and
antilinear forms in \eqref{eq:Phi} and \eqref{eq:rhs}, respectively.

We begin by stating the Bramble-Hilbert lemma
(\emph{cf.}~\cite[Theorem 4.1.3]{Ciarlet:2002aa}), which shall be
required to give error estimates to the approximation of exact
integration by numerical quadrature.

\begin{lemma}[Bramble-Hilbert]\label{lemma:BHvec}
Let $q\in\IN$ and $\cO$ be an open subset of $\IR^q$ with a
Lipschitz-continuous boundary. For some integer $k\geq 0$
and $p\in [1,\infty]$, let $\bm{f}$ be a continuous linear form
on $\bm{W}^{k+1,p}(\cO)$ with the property that
\begin{align*}
\bm{f}(\bm{q})=0\quad \forall\;\bm{q}\in \mathbb{P}_k(\cO;\IC^3).
\end{align*}
Then, there exists a constant $C_{\cO}$, depending on the domain
such that for all $\bV\in\bm{W}^{k+1,p}(\cO)$,
\begin{align*}
\modulo{\bm{f}(\bV)}
\leq C_{\cO}\norm{\bm{f}}{(\bm{W}^{k+1,p}(\cO))'}
\modulo{\bV}_{\bm{W}^{k+1,p}(\cO)}.
\end{align*}
\end{lemma}

The following Lemmas provide local error estimates
for the quadrature rules over arbitrary tetrahedrons $K\in\tau_h$.
Their proofs are analogous to those in \cite[Chapter 4]{Ciarlet:2002aa}
for grad-conforming finite elements.

\begin{lemma}\label{lem:ConErrK}
Let $K\in\tau_h$, $m\in\IN$ and $\bM=(M_{i,j})_{i,j=1}^3$, with
$\bM(\bx)\in\IC^{3\times 3}$, be such that
$M_{i,j}\in W^{m,\infty}(K)$ for all $i,j\in\{1,2,3\}$. If $Q_{\breve{K}}$ is
a quadrature rule as in Definition \ref{def:quad} such that it
is exact for polynomials of degree ${k+m-1}$, then the local
quadrature error (for $Q_K$ as in \eqref{eq:QK})
\begin{equation*}
\cE_K(\bM\bU_h\cdot\bV_h):=\int_{K}\bM\bU_h\cdot\bV_h\d\! \bx
-Q_{K}\left(\bM\bU_h
\cdot\bV_h\right),
\end{equation*}
is such that for all $\bU_h$, $\bV_h \in\mathbb{P}_{k}({K};\IC^{3})$
\begin{align}
\modulo{\cE_K(\bM\bU_h\cdot\bV_h)}\leq
CC_{\bM}h^{m}\norm{\bU_h}{m,K}\norm{\bV_h}{0,K},\label{eq:ErrorRes}
\end{align}
for a positive constant $C$ independent of $h$, $K$ and $\bM$ and
$$C_{\bM}:=\sum_{i,j=1}^{3}\norm{M_{i,j}}{W^{m,\infty}(K)}.$$
\end{lemma}

\begin{proof}
Let $\bm{\phi}\in \bm{W}^{m,\infty}(\breve{K})$ and
$\bV_h\in \mathbb{P}_{k}(\breve K;\IC^{3})$, then,
\begin{align*}
  \modulo{\mathcal{E}_{\breve{K}}(\bm{\phi}\cdot\bV)}&\leq
  C_{\mathcal{E}}\norm{\bm{\phi}\cdot \bV_h}{L^\infty(\breve{K})}\leq
  C_{\mathcal{E}}\norm{\bm{\phi}}{\bm{L}^\infty(\breve{K})}
  \norm{\bV_h}{\bm{L}^\infty(\breve{K})}\leq
  C_{\mathcal{E}}\norm{\bm{\phi}}{\bm{W}^{m,\infty}(\breve{K})}
  \norm{\bV_h}{0,\breve{K}},
\end{align*}
for some positive $C_{\mathcal{E}}$---depending only on
$\breve{K}$---where the last inequality follows from norm equivalence over
$\mathbb{P}_{k}(K;\IC^{3})$. For a fixed
$\bV_h\in \mathbb{P}_{k}(K;\IC^{3})$ the form
$\mathcal{E}_{\breve{K}}(\bm{\phi}\cdot\bV_h)$ is linear and bounded on
$\bm\phi\in\bm{W}^{m,\infty}(\breve{K})$ and satisfies
$\mathcal{E}_{\breve{K}}(\bm{\phi}\cdot\bV_h)=0$ for all
$\bm{\phi}\in\mathbb{P}_{m-1}(\breve{K};\IC^3)$.  By the Bramble-Hilbert Lemma
(Lemma \ref{lemma:BHvec}) there exists a positive constant $C_{\breve{K}}$ such that
\begin{align*}
  \modulo{\mathcal{E}_{\breve{K}}(\bm{\phi}\cdot\bV_h)}
  \leq C_{\breve{K}}\seminorm{\bm{\phi}}{\bm{W}^{m,\infty}(\breve{K})}
  \norm{\bV_h}{0,\breve{K}}.
\end{align*} 
Then, for any $K\in\tau_h$ and $\bU_h$, $\bV_h\in\mathbb{P}_{k}(K;\IC^{3})$,
\begin{align*}
  \modulo{\mathcal{E}_{K}(\bM\bU_h\cdot\bV_h)}
  \leq C_{\breve{K}}\modulo{\det{(\mathbb{J}_K)}}
  \seminorm{(\bM\bU_h)\circ T_K}{W^{m,\infty}(\breve{K})}
  \norm{\bV_h\circ T_K}{0,\breve{K}}.
\end{align*}
We begin by bounding
$\seminorm{(\bM\bU_h)\circ T_K}{W^{m,\infty}(\breve{K})}$. 
\begin{align}
  \seminorm{(\bM\bU_h)\circ T_K}{W^{m,\infty}(\breve{K})}&\leq
  c\sum_{i=1}^{3}\seminorm{\sum_{j=1}^{3}(M_{i,j}\circ T_K)({\bU}_{h,j}\circ T_K)}
  {W^{m,\infty}(\breve{K})},\nonumber\\
  &\leq c\sum_{i,j=1}^{3}\seminorm{(M_{i,j}\circ T_K)({\bU}_{h,j}\circ T_K)}
  {W^{m,\infty}(\breve{K})},\nonumber\\
  &\leq c\sum_{i,j=1}^{3}\sum_{n=0}^{m}\seminorm{\psi^g_K(M_{i,j})}
  {W^{m-n,\infty}(\breve{K})}\seminorm{\psi^g_K({\bU}_{h,j})}{W^{n,\infty}(\breve{K})},
  \nonumber\\
  &\leq c\sum_{i,j=1}^{3}\sum_{n=0}^{m}\seminorm{\psi^g_K(M_{i,j})}
  {W^{m-n,\infty}(\breve{K})}\seminorm{\psi^g_K({\bU}_{h,j})}{n,\breve{K}},
  \label{eq:phiequiv}\\
  &\leq c\norm{\mathbb{J}_{K}}{3\times 3}^{m}
  \modulo{\det{(\mathbb{J}_K)}}^{-\frac{1}{2}}
  \sum_{i,j=1}^{3}\sum_{n=0}^{m}\seminorm{M_{i,j}}
  {W^{m-n,\infty}({K})}\seminorm{{\bU}_{h,j}}{n,{K}},
  \label{eq:hk}\\
  &\leq c\norm{\mathbb{J}_{K}}{3\times 3}^{m}
  \modulo{\det{(\mathbb{J}_K)}}^{-\frac{1}{2}}\norm{\bU_h}{m,{K}}
  \sum_{i,j=1}^{3}\norm{M_{i,j}}{W^{m,\infty}(K)},\label{eq:auxs2}
\end{align}
where the positive constant $c$ is independent of
$K$ and may change at each step,
\eqref{eq:phiequiv} employs the equivalence of norms
in spaces of finite dimension and \eqref{eq:hk} is a
consequence of \eqref{eq:struct1K}. A similar bound
for $ \norm{\bV_h\circ T_K}{0,\breve{K}}$ may be
obtained analogously:
\begin{align}
\norm{\bV_h\circ T_K}{0,\breve{K}}\leq
c\modulo{\det{(\mathbb{J}_K)}}^{-\frac{1}{2}}\norm{\bV_h}{0,K},\label{eq:auxs3}
\end{align}
for a positive constant $c$ as before. Then,
\begin{align*}
  \modulo{\mathcal{E}_{K}(\bM\bU_h\cdot\bV_h)}
  &\leq C_{\breve{K}}\modulo{\det{(\mathbb{J}_K)}}
  \seminorm{(\bM\bU_h)\circ T_K}{W^{m,\infty}(\breve{K})}
  \norm{\bV_h\circ T_K}{0,\breve{K}}\\
  &\leq C_{\breve{K}}C_{\bM}c\norm{\mathbb{J}_{K}}{3\times 3}^{m}
  \norm{\bU_h}{m,K}
  \norm{\bV_h}{0,{K}}
  \leq CC_{\bM}h^{m}
  \norm{\bU_h}{m,K}
  \norm{\bV_h}{0,{K}},
\end{align*}
where $c$ is a positive constant independent of $K$ and $M$
and $C$ follows from combining $c$ and $C_\rK$.
\end{proof}

\begin{lemma}\label{lem:ConErrKrhs}
Let $K\in\tau_h$, $m\in\IN$ and $q\in\IR$ such that
\begin{align}
q\geq 2\quad\mbox{and}\quad q>\frac{3}{m},\label{eq:q1cond}
\end{align}
and $Q_{\breve{K}}$ be a quadrature rule as in Definition \ref{def:quad}
such that it is exact on polynomials of degree ${k+m-1}$.
Then, if $\bJ\in\bW^{m,q}(K)$,
the local quadrature error $\cE_K(\bJ\cdot\bV_h)$
(as defined in Lemma \ref{lem:ConErrK}) is such that for all
$\bV_h\in\mathbf{P}_{k}({K};\IC^{3})$
\begin{align*}
\modulo{\cE_K(\bJ\cdot\bV_h)}\leq
Ch^{m}\modulo{K}^{\frac{1}{2}-\frac{1}{q}}\norm{\bJ}{\bW^{m,q}(K)}\norm{\bV_h}{0,K},
\end{align*}
for a positive constant $C$ independent of $h$ and $K$.
\end{lemma}
\begin{proof}
The proof is exactly as that of Lemma \ref{lem:ConErrK} upon realizing that, thanks
to \eqref{eq:q1cond}, it holds
\begin{align*}
\norm{\bJ\circ T_{K}}{\Lp{\infty}{\rK}}
\leq c\norm{\bJ\circ T_{K}}{\bW^{m,q}(K)}
\end{align*}
for some positive $c$ independent of $K$ and the meshsize
(\emph{cf.}~\cite[Thm.~2.5]{Steinbach:2007aa} or \cite[Thm.~3.6]{Monk:2003aa}).
\end{proof}

With the previous estimates at hand, we can now prove
the following Theorems providing the necessary conditions for us
to employ Strang's Lemma (Lemma \ref{lemma:Strang}) in the
proof of Theorem \ref{thm:mainresaffine}.

\begin{theorem}[Consistency error for the sesquilinear form]
\label{thm:consesq}
Recall $k\in\IN$ as the polynomial degree of our approximation
spaces. Let $m\in\IN$ and assume the following of the quadrature
rules defining $\widetilde\Phi_{h_i}$ in Definition \ref{def:numsesqant}:
\begin{itemize}
\item The quadrature rule $Q^1_{\rK}$ is exact for polynomials
of degree $k+m-2$.
\item The quadrature rule $Q^2_{\rK}$ is exact for polynomials
of degree $k+m-1$.
\end{itemize}
Then, under Assumptions \ref{ass:polydom} and \ref{ass:meshPoly} and if the
coefficients of $\mu^{-1}$ and $\epsilon$ belong to
$W^{m,\infty}(\D)$, 
\begin{align*}
&\modulo{\Phi(\bU_{h_i},\bV_{h_i})-\widetilde{\Phi}_{h_i}(\bU_{h_i},\bV_{h_i})}\\
&\leq C_{\Phi}{h_i}^{m} \sum_{K\in\tau_{h_i}}\left(C_{\mu^{-1}}\norm{\curl\bU_{h_i}}{m,K}\norm{\curl\bV_{h_i}}{0,K}
+\omega^2C_{\epsilon}\norm{\bU_{h_i}}{m,K}\norm{\bV_{h_i}}{0,K}\right)
\end{align*}
for all $\bU_{h_i}$, $\bV_{h_i}\in\bm{P}_0^c(\tau_{h_i})$, where $C_{\mu^{-1}}$ and
$C_{\epsilon}$ are positive constants depending on $\mu^{-1}$ and
$\epsilon$, and $C_{\Phi}$ is a positive constant independent of the mesh sizes $\{h_i\}_{i\in\IN}$.
\end{theorem}
\begin{proof}

The result comes from noticing that, for
all $\bU_h$ and $\bV_h\in\bm P^c_0(\tau_h)$
and all $K\in\tau_h$,
\begin{gather*}
\bV_h\vert_{K}\in\mathbb{P}_{k}(K;\IC^3),\quad\curl\bV_h\vert_{K}\in\mathbb{P}_{k-1}(K;\IC^3),\quad\mbox{and}\\
\modulo{\Phi(\bU_h,\bV_h)-\widetilde{\Phi}_h(\bU_h,\bV_h)}
\leq\sum_{K\in\tau_h}\modulo{\cE_K(\mu^{-1}\curl\bU_h\cdot\curl\bV_h)}
+\omega^2\modulo{\cE_K(\epsilon\bU_h\cdot\bV_h)},
\end{gather*}
where $\cE_K(\cdot,\cdot)$ is as in Lemma \ref{lem:ConErrK} and 
\begin{align*}
C_{\mu^{-1}}:=\sum_{i,j=1}^3\norm{(\mu^{-1})_{i,j}}{W^{m,\infty}(D)}\quad\mbox{and}\quad
C_{\epsilon}:=\sum_{i,j=1}^3\norm{\epsilon_{i,j}}{W^{m,\infty}(D)}.
\end{align*}
\end{proof}
\begin{theorem}[Consistency error for the antilinear form]
\label{thm:conant}
Recall $k\in\IN$ as the polynomial degree of our approximation
spaces. Let $m\in\IN$ and assume the quadrature rule $Q^3_{\breve{K}}$
from Definition \ref{def:numsesqant} is exact for polynomials of degree
$k+m-1$. Then, under Assumptions
\ref{ass:polydom} and \ref{ass:meshPoly} and if $\bJ$ is such that
$\bJ\in \bW^{m,q}(D)$ for some $q\in\IR$ such that $q > \frac{3}{m}$
and $q\geq 2$, then
\begin{align*}
\modulo{\bF(\bV_{h_i})-\widetilde{\bF}(\bV_{h_i})}\leq
C_{\bF} \omega {h_i}^{m} \modulo{\D}^{\frac{1}{2}-\frac{1}{q}}
\norm{\bJ}{\bW^{m,q}(\D)}\norm{\bV_{h_i}}{0,\D}
\end{align*}
for all $\bV_{h_i}\in\bm{P}_0^c(\tau_{h_i})$, for a positive constant
$C_{\bF}$ independent of the mesh sizes $\{h_i\}_{i\in\IN}$.
\end{theorem}
\begin{proof}
The result follows analogously from that of Theorem \ref{thm:consesq}
and Hölder's inequality.
\end{proof}
We are now ready to present a proof for the main result of this
section.
\begin{proof}[Proof of Theorem \ref{thm:mainresaffine}]
From our assumptions and Theorem \ref{thm:consesq} we can employ 
Strang's Lemma (Lemma \ref{lemma:Strang}), since
\begin{align*}
&\modulo{\Phi(\bU_{h_i},\bV_{h_i})-\widetilde{\Phi}_{h_i}(\bU_{h_i},\bV_{h_i})}\\
&\leq C_{{\Phi}}{h_i}^{\lceil r\rceil} \sum_{K\in\tau_{h_i}}\left(C_{\mu^{-1}}\norm{\curl\bU_{h_i}}{{\lceil r\rceil},K}\norm{\curl\bV_{h_i}}{0,K}
+C_{\epsilon}\norm{\bU_{h_i}}{{\lceil r\rceil},K}\norm{\bV_{h_i}}{0,K}\right)\\
&\leq C_{{\Phi}}(C_{\mu^{-1}}+C_{\epsilon}){h_i}^{\lceil r\rceil} \left(\norm{\curl\bU_{h_i}}{{\lceil r\rceil},\D}\norm{\curl\bV_{h_i}}{0,\D}
+\norm{\bU_{h_i}}{{\lceil r\rceil},\D}\norm{\bV_{h_i}}{0,\D}\right)\\
&\leq C_{{\Phi}}(C_{\mu^{-1}}+C_{\epsilon}){h_i}^{\lceil r\rceil} \norm{\bU_{h_i}}{\hscurl{\D}{{\lceil r\rceil}}}\norm{\bV_{h_i}}{\hcurl{\D}}.
\end{align*}
Hence, there is some $\ell\in\IN$ so that for all
$i\in\IN$ with $i\geq \ell$ there exists a
unique solution to Problem \ref{prob:varprobnum} 
$\widetilde\bE_{h_i}\in\bm{P}^c_0(\tau_{h_i})$ that satisfies
\begin{align*}
\norm{\bE-\widetilde\bE_{h_i}}{\hcurl{D}}\leq C_{S}\Bigg{(}\norm{\bE-\bE_{h_i}}{\hcurl{\D}}+\sup_{\bV_{h_i}\in\bm P^c_0(\tau_{h_i})}\frac{\vert{\Phi(\bE_{h_i},\bV_{h_i})-\widetilde\Phi_{h_i}(\bE_{h_i},\bV_{h_i})}\vert}{\norm{\bV_{h_i}}{\hcurl{\D}}}\\
 +C_{\bF}\omega h_i^{\lceil r\rceil}\modulo{\D}^{\frac{1}{2}-\frac{1}{q}}\norm{\bJ}{\bW^{\lceil r \rceil,q}(\D)} \Bigg{)},
\end{align*}
where $C_S$ follows from Strang's Lemma, $C_{\bF}$ follows
from Theorem \ref{thm:conant} and $\bE_{h_i}\in\bm P^c_0(\tau_{h_i})$ is the
unique discrete solution to Problem \ref{prob:varprobdis}.
From \cite[Thm.~3.3]{ern2018analysis} we see that
\begin{align*}
\norm{\bE-\bE_{h_i}}{\hcurl{\D}}\leq c\;h_i^r\norm{\bE}{\hscurl{\D}{r}},
\end{align*}
for some positive constant $c>0$ independent of the mesh-size.
We continue by bounding the error between $\Phi$ and $\widetilde{\Phi}_{h_i}$,
first noticing that for any $\bV_{h_i}\in\bm P^c_0(\tau_{h_i})$, there holds
\begin{align*}
&\vert{\Phi(\bE_{h_i},\bV_{h_i})-\widetilde\Phi_{h_i}(\bE_{h_i},\bV_{h_i})}\vert\\
&\leq C_{{\Phi}}h^{\lceil r\rceil}_i \sum_{K\in\tau_{h_i}}\left(C_{\mu^{-1}}\norm{\curl\bE_{h_i}}{{\lceil r\rceil},K}\norm{\curl\bV_{h_i}}{0,K}
+\omega^2C_{\epsilon}\norm{\bE_{h_i}}{{\lceil r\rceil},K}\norm{\bV_{h_i}}{0,K}\right).
\end{align*}
For arbitrary $K\in\tau_{h_i}$ and sequentially employing Lemmas
\ref{lemma:quasiInt}, \ref{lemma:IBound} and \ref{lemma:inverseineq}, it holds
\begin{align*}
\norm{\bE_{h_i}}{\lceil r \rceil,K}&\leq\norm{\bE_{h_i}-\mathcal{I}^{\#,c}_{K}(\bE)}{\lceil r \rceil,K}+\norm{\mathcal{I}^{\#,c}_{K}(\bE)}{\lceil r \rceil,K}\\
&\leq\norm{\mathcal{I}^{\#,c}_{K}(\bE_{h_i}-\bE)}{\lceil r \rceil,K}+c_{1}h_i^{r-\lceil r \rceil}\norm{\bE}{ r ,K}\\
&\leq c_2h_{i}^{-\lceil r \rceil}\norm{\mathcal{I}^{\#,c}_{K}(\bE_{h_i}-\bE)}{0,K}+c_1h_i^{r-\lceil r \rceil}\norm{\bE}{ r ,K}\\
&\leq c_2h_{i}^{-\lceil r \rceil}\norm{(\bE_{h_i}-\bE)}{0,K}+c_1h_i^{r-\lceil r \rceil}\norm{\bE}{ r ,K},\\
&\leq c\left(h_{i}^{-\lceil r \rceil}\norm{(\bE_{h_i}-\bE)}{0,K}+h_i^{r-\lceil r \rceil}\norm{\bE}{ r ,K}\right),
\end{align*}
where $c=c_1+c_2$ and the positive constants
$c_1$ and $c_2$ come from the previously referenced
Lemmas and are independent of both the mesh-size and $r\leq k$.
Analogously and since $\curl\bE_{h_i}\in\bm P^d(\tau_{h_i})$
\cite[Lemma 5.40]{Monk:2003aa}, we have
\begin{align*}
\norm{\curl\bE_{h_i}}{\lceil r \rceil,K}
\leq c\left(h_{i}^{-\lceil r \rceil}\norm{\curl(\bE_{h_i}-\bE)}{0,K}+h_i^{r-\lceil r \rceil}\norm{\curl\bE}{ r ,K}\right),
\end{align*}
for some positive $c$ independent of $h$ and $r$.
Therefore,
\begin{align*}
&\vert{\Phi(\bE_{h_i},\bV_{h_i})-\widetilde\Phi_{h_i}(\bE_{h_i},\bV_{h_i})}\vert\\
&\leq C_{{\Phi}}(C_{\mu^{-1}}+\omega^2C_{\epsilon}){h_i}^{\lceil r\rceil} \sum_{K\in\tau_{h_i}}\left(\norm{\curl\bE_{h_i}}{{\lceil r\rceil},K}\norm{\curl\bV_{h_i}}{0,K}
+\norm{\bE_{h_i}}{{\lceil r\rceil},K}\norm{\bV_{h_i}}{0,K}\right)\\
&\leq C_{{\Phi}}(C_{\mu^{-1}}+\omega^2C_{\epsilon})c\; {h_i}^{r} \sum_{K\in\tau_{h_i}}\left(\norm{\curl\bE}{{\lceil r\rceil},K}\norm{\curl\bV_{h_i}}{0,K}
+\norm{\bE}{{\lceil r\rceil},K}\norm{\bV_{h_i}}{0,K}\right)\\
& +C_{{\Phi}}(C_{\mu^{-1}}+\omega^2C_{\epsilon})c\sum_{K\in\tau_{h_i}}\left(\norm{\curl(\bE-\bE_{h_i})}{0,K}\norm{\curl\bV_{h_i}}{0,K}
+\norm{\bE-\bE_{h_i}}{0,K}\norm{\bV_{h_i}}{0,K}\right)\\
&\leq C_{{\Phi}}(C_{\mu^{-1}}+\omega^2C_{\epsilon})c\; {h_i}^r\norm{\bE}{\hscurl{\D}{r}}\norm{\bV_{h_i}}{\hcurl{\D}},
\end{align*}
where the last inequality follows from \cite[Thm.~3.3]{ern2018analysis} and
the positive constant $c$, independent of $h$ and $r$, may
vary at each step. Finally,
\begin{align*}
\norm{\bE-\widetilde\bE_{h_i}}{\hcurl{D}}\leq C_{S}\left( c\;{h_i}^r\norm{\bE}{\hscurl{\D}{r}}+C_{{\Phi}}(C_{\mu^{-1}}+\omega^2C_{\epsilon})c\; {h_i}^r\norm{\bE}{\hscurl{\D}{r}}\right.\\
+\left.C_{\bF}\omega {h_i}^{\lceil r\rceil}\modulo{\D}^{\frac{1}{2}-\frac{1}{q}}\norm{\bJ}{\bW^{\lceil r \rceil,q}(\D)}\right)
\end{align*}
as stated.
\end{proof}

\section{Finite Elements and Consistency Error Estimates for Smooth Curved Domains}
\label{sec:CurvDom}
We now drop the requirement that $\D$ be polyhedral
(Assumption \ref{ass:polydom}). As a direct consequence
of this, it will prove impractical to generate meshes that
cover $\D$ exactly and we shall instead consider a
sequence of meshes that approximate $\D$ as the
mesh-size $h$ decreases.

\begin{assumption}
\label{ass:ContDom}
The bounded domain $\D$ is of class $\C^{\mathfrak{M}}$ for some ${\mathfrak{M}}\in\IN$.
\end{assumption}

\subsection{Finite elements}
We begin by introducing a sequence of polyhedral meshes
constructed from disjoint, matching tetrahedrons that
approximate $\D$, $\{\tau_{h_i}\}_{i\in\IN}$. As in
\cite{hernandez2003finite,lenoir1986} and
\cite[Section 1.3.2]{ern2004theory}, we will require some
assumptions from $\{\tau_{h_i}\}_{i\in\IN}$.
\begin{assumption}[Assumptions on the polyhedral meshes.]
\label{ass:meshPolyCur}
$\{\tau_{h_i}\}_{i\in\IN}$ is a sequence of affine and
quasi uniform meshes. The nodes of its boundary are
located in $\Gamma$ and the polyhedral
domain generated by each mesh, denoted
$\D^{\text{poly}}_{h_i}$ (with boundary $\Gamma^{\text{poly}}_{h_i}$), approximates $\D$ so that
$$\lim_{i\rightarrow \infty}\dist(\D,\D^{\text{poly}}_{h_i})=0.$$
\end{assumption}

As in the previous Section,
the elements of the meshes $\{\tau_{h_i}\}_{i\in\IN}$ may be
constructed through affine transformations as in Definition
\ref{def:affinebijectivemap}. We continue by introducing
curved meshes that approximate $\D$, which shall be constructed
from the polyhedral meshes $\{\tau_{h_i}\}_{i\in\IN}$.

\begin{definition}[Approximated meshes]
For each polyhedral mesh $\tau_h\in\{\tau_{h_i}\}_{i\in\IN}$,
we consider $\widetilde\tau_h$ to be the \emph{approximated mesh},
which shares its nodes with $\tau_h$, but is composed of
curved tetrahedrons that cover a domain
$\D_h$ (that approximates $\D$) exactly.
For a given $K\in\tau_h$ we refer to the element of
$\widetilde\tau_h$ that shares its nodes with $K$
as $\widetilde K$ and consider the bijective mappings
$T_{\widetilde K}:\breve{K}\mapsto\widetilde K$ 
to be polynomial.
\end{definition}

Henceforth, let $\tau_h$ and $\widetilde{\tau}_h$ be an arbitrary
meshes in $\{\tau_{h_i}\}_{i\in\IN}$ and $\{\widetilde\tau_{h_i}\}_{i\in\IN}$,
respectively.

All numerical computations are to be done on the approximated
meshes $\{\widetilde\tau_{h_i}\}_{i\in\IN}$. As such, we shall require
numerical approximations of $\Phi$ and $\bF$ on the approximated
meshes. Also notice that $\D_{h}\setminus\D$ need not be empty,
so we will also require to assume $\mu^{-1}$, $\epsilon$
and $\bJ$ to be well defined outside of $\D$.

\begin{assumption}\label{ass:holdall}
There exists a bounded domain
$\D^H\subset\IR^3$ of class
$C^{\infty}$, hold-all domain, such that
\begin{align*}
\D\subset\D^H\quad\mbox{and}\quad \D_{h_i}\subset\D^H\quad\forall\; i\in\IN.
\end{align*} 
Moreover, $\mu^{-1}$, $\epsilon$ and $\bJ$ belong to $\C^0(\D_H)$.
\end{assumption}

\begin{assumption}[Assumptions on the approximated and exact elements]
\label{ass:k_gregular}
Let $\mathfrak{K}\in\IN$ with $\mathfrak{K}<\mathfrak{M}$
from Assumption \ref{ass:ContDom}.
The family of approximate meshes
$\{\widetilde\tau_{h_i}\}_{i\in\IN}$ is assumed to be
$\mathfrak{K}$-regular, i.e.~the mappings ${T}_{\widetilde K}$,
$\widetilde K\in\widetilde\tau_h$, are $\C^{\mathfrak{K}+1}$-diffeomorphisms
that belong to $\mathbb{P}_{\mathfrak{K}}(\rK;\widetilde{K})$. Also,
the following bounds for derivatives of these transformations
hold for all $\widetilde{K}\in\widetilde\tau_h$
\begin{align*}
\sup_{\bx\in\breve{K}}\norm{D^{n}{T}_{\widetilde K}(\bx)}{}\leq C_{n}h^{n}
\quad\mbox{and}\quad\sup_{\bx\in\breve{K}}\norm{D^{n}\left({T}_{\widetilde{K}}^{-1}\right)(\bx)}{}
\leq C_{-n}h^{-n}\quad \forall\; n\in\{1,\hdots,\mathfrak{K}+1\},
\end{align*}
where $C_n$ and $C_{-n}$ are positive constants independent from
the mesh-size for all integers $n\leq \mathfrak{K}+1$, $D^{n}{T}_K$
is the Frechet derivative of order $n$ of $T_K$ and
$\Vert{D^{n}\widetilde{T}_K(\bx)}\Vert$ is the appropriate
induced norm, with functional spaces omitted for the sake of brevity.
Furthermore, we assume $$\det\mathbb{J}_{\widetilde K}(\bx)>0,$$ for all $\bx\in\rK$
and that there exists some positive $\theta\in\IR$,
independent of the mesh-size, such that for all $\widetilde K\in\widetilde\tau_h$
\begin{align*}
\frac{1}{\theta}\leq{\frac{\det \mathbb{J}_{\widetilde K}(\bx)}{\det \mathbb{J}_{\widetilde K}(\by)}}\leq \theta\quad\forall\;\bx\;,\by\in \rK.
\end{align*}
\end{assumption}

In Assumption \ref{ass:k_gregular}, $\mathfrak{K}$ represents the degree
of the polynomial approximation of $\D$. We expect the
rate with which the sequence of domains $\{\D_{h_i}\}_{i\in\IN}$
converges to $\D$ to be dependent on $\mathfrak{K}$,
so that larger $\mathfrak{K}$ will imply faster convergence, e.g.
\begin{align*}
\dist(\D,\D_{h_i})\leq C_{\mathfrak{K}}h_i^{f(\mathfrak{K})},
\end{align*}
for some strictly increasing positive function $f:\IN\to\IR$ and
a positive constant $C_{\mathfrak{K}}$ that may depend on $\mathfrak{K}$. Furthermore, notice that for any curved tetrahedron $\widetilde K$
and any $\by\in\rK$
\begin{align}
\label{eq:detKcurv}
\modulo{\widetilde K}=\int_{\widetilde K}1\d\!\bx
=\int_{\rK}\det{\mathbb{J}_{\widetilde K}(\bx)}\d\!\bx .
\end{align}
\begin{lemma}\label{lem:detJmodK}
Under Assumption \ref{ass:k_gregular} and for any
$\by\in\rK$, it holds
\begin{align*}
\frac{1}{\theta}\leq\frac{\modulo{\widetilde K}}{\modulo{\rK}\det\mathbb{J}_{\widetilde{K}}(\by)}\leq\theta.
\end{align*}
\end{lemma}
\begin{proof}
Notice that for any $\by\in\rK$ 
\begin{align*}
\int_{\rK}\det{\mathbb{J}_{\widetilde K}(\bx)}\d\!\bx=
\det{\mathbb{J}_{\widetilde K}(\by)}\int_{\rK}\frac{\det{\mathbb{J}_{\widetilde K}(\bx)}}{\det{\mathbb{J}_{\widetilde K}(\by)}}\d\!\bx.
\end{align*}
The proof then follows from our assumed bounds for
$\dfrac{\det{\mathbb{J}_{\widetilde K}(\bx)}}{\det{\mathbb{J}_{\widetilde K}(\by)}}$
and \eqref{eq:detKcurv}.
\end{proof}

As before, we consider finite elements on curved tetrahedrons
$\widetilde K\in\widetilde\tau_h$ as triples\\
$(\widetilde K,P_{\widetilde K},\Sigma_{\widetilde K})$,
so that Definition \ref{def:femtriple} remains valid on 
curved tetrahedrons. We define the curl-conforming
element on a curved tetrahedron $\widetilde K$ as
\begin{align*}
{\bm{P}}^c_{\widetilde K}&:=\{{\bm{p}}\; :\; \mathbb{J}_{\widetilde K}^{\top}(\bm{p}\circ {T}_{\widetilde{K}})\in\bm{P}^c_{\breve K}\}.
\end{align*}
The function spaces for grad- and
div-conforming finite elements are defined in a similar manner
\begin{align*}
  {P}^g_{\widetilde K}:=\{p\; :\; p\circ {T}_{\widetilde{K}}\in P^g_{\breve K}\},\qquad
  {\bm{P}}^d_{\widetilde K}:=\{\bm{p}\; :\; \det(\mathbb{J}_{\widetilde{K}})\mathbb{J}^{-1}_{\widetilde K}(\bm{p}\circ{T}_{\widetilde{K}})\in\bm{P}^d_{\breve K}\}.
\end{align*}
Discrete spaces on curved meshes are then defined as in the previous section
\begin{gather*}
P^g(\widetilde\tau_h):= \left\lbrace v_h\in\hsob{1}{\D} \ :\ 
v_h\vert_{\widetilde K}\in P^g_{\widetilde K}\quad\forall\; \widetilde K\in\widetilde\tau_h\right\rbrace,\;
P^g_0(\widetilde\tau_h):=  P^g(\widetilde\tau_h)\cap\hosob{1}{\D},\\
\bm P^c(\widetilde\tau_h):= \left\lbrace \bV_h\in\hcurl{\D} \ :\  
\bV_h\vert_{\widetilde K} \in {\bm{P}}^c_{\widetilde K}\quad\forall\; \widetilde K\in\widetilde\tau_h\right\rbrace,\; 
\bm P^c_0(\widetilde\tau_h):= \bm P^c(\widetilde\tau_h)\cap\hocurl{\D},\\
\bm P^d(\widetilde\tau_h):= \left\lbrace \bV_h\in\hdiv{\D} \ :\  
\bV_h\vert_{\widetilde K} \in {\bm{P}}^d_{\widetilde K}\quad\forall\; \widetilde K\in\widetilde\tau_h\right\rbrace,
\; \bm P^d_0(\widetilde\tau_h):= \bm P^d(\widetilde\tau_h)\cap\hodiv{\D},\\
P^b(\widetilde\tau_h):= \left\lbrace v_h\in\lp{2}{\D} \ :\ 
v_h\vert_{\widetilde K}\in {P}^b_{\widetilde K}\quad\forall\; \widetilde K\in\widetilde\tau_h  \right\rbrace.
\end{gather*}
Pullbacks from functions defined on curved tetrahedrons and triangles
are defined analogously as those in \eqref{eq:pullbacks}.
We continue by stating a property analogous to that in \eqref{eq:struct2K}
in the context of curved meshes.
\begin{lemma}[Lemma 1 in \cite{ciarlet1972combined}]\label{lem:curvineqK}
Let $p\in\IR$ and $l\in\IN_{0}$ be such that $p\geq 1$
and $l\leq \mathfrak{K}+1$.
Then, for a given $\breve{u}\in W^{s,q}(\breve K)$ and a curved
tetrahedron $\widetilde{K}\in\widetilde{\tau}_h$, the function
$u$ defined as
\begin{align*}
u:={\psi^g_{\widetilde{K}}}^{-1}(\breve{u})=\breve{u}\circ T_{\widetilde{K}}^{-1}
\end{align*}
 belongs to $W^{l,p}( \widetilde{K})$ and
 \begin{align*}
 \seminorm{\breve u}{W^{l,p}(\breve K)}\leq
 C\inf_{\bx\in\breve{K}}\modulo{\det\mathbb{J}_{\widetilde{K}}(\bx)}^{-\frac{1}{p}}h^{l}
 \norm{u}{W^{l,p}(\widetilde{K})},
 \end{align*}
 for a positive constant $C$ independent of $\widetilde K$ and
 the mesh-size.
\end{lemma}
\subsection{Discrete variational problem}
We continue by introducing appropriate modifications of the
sesquilinear and antilinear forms considered in the previous sections.
In particular, for each $i\in\IN$ and all $\bU$, $\bV\in\hcurl{\D_{h_i}}$, we shall consider
\begin{align}
\label{eq:Phih}
\Phi_{h_i}(\bU,\bV)&:=
\int_{\D_{h_i}} \mu^{-1} \curl \mathbf{U} \cdot \curl \overline{\mathbf{V}}-
\omega^2\epsilon \mathbf{U} \cdot \overline{\mathbf{V}} \d\!\bx,\\
\label{eq:rhsh}
\bF_{h_i}(\bV)&:=-\imath\omega\int_{\D_{h_i}}\bJ\cdot\overline{\bV}\d\!\bx.
\end{align}

\begin{problem}[Discrete variational problem on curved meshes]
\label{prob:varprobdiscurv}%
Find $\bE_h\in \bm P^c_0(\widetilde\tau_h)$ such that
\begin{align*}
\Phi_h(\bE_h,\bV_h)=\bF_h(\bV_h),
\end{align*}
for all $\bV_h \in \bm P^c_0(\widetilde\tau_h)$.
\end{problem}

As with Assumptions \ref{ass:sesqform} and \ref{ass:sesqformdis},
we assume our framework to be such
that a unique discrete solution $\bE_{h}\in\bm P^c_0(\widetilde\tau_{h})$
exists for all meshes in $\{\widetilde\tau_{h_i}\}_{i\in\IN}$.

\begin{assumption}[Wellposedness on $\bm P^c_0(\widetilde\tau_h)$]
\label{ass:sesqformdiscurv}
We assume the sesquilinear forms $\{\Phi_{h_i}\}_{i\in\IN}$ in \eqref{eq:Phih}
to satisfy the following:
\begin{gather*}
\sup_{\bU_{h_i}\in \bm P^c_0(\widetilde\tau_{h_i})\setminus\{\bnul\}}\modulo{\Phi_{h_i}(\bU_{h_i},\bV_{h_i})}>\alpha>0\quad \forall\;\bV_{h_i}\in\bm P^c_0(\widetilde\tau_{h_i})\setminus\{\bnul\},\\
\inf_{\bU_{h_i}\in\bm P^c_0(\widetilde\tau_{h_i})\setminus\{\bnul\}}\left(\sup_{\bV_{h_i}\in \bm P^c_0(\tau_{h_i})\setminus\{\bnul\}} 
\frac{\modulo{\Phi_{h_i}(\bU_{h_i},\bV_{h_i})}}{\norm{\bU_{h_i}}{\hcurl{\D_{h_i}}}\norm{\bV_{h_i}}{\hcurl{\D_{h_i}}}}\right)\geq C>0,
\end{gather*}
for all $i\in\IN$.
\end{assumption}
\subsection{Numerical integration}
\label{sec:numintCurv}
We follow as in the previous section and consider numerical computation
of the integrals in Problem \ref{prob:varprobdis}. We recall the definition of our
quadrature rule over $\rK$ in Definition \ref{def:quad}
and introduce numerical integration on curved elements
$\widetilde K\in\widetilde\tau_h$ as
\begin{align}\label{eq:Qcurv}
  Q_{\widetilde K}(\phi):=\sum_{l=1}^L w_{l,\widetilde K} \phi(\bm{b}_{l,\widetilde K})\quad\text{where}\quad
  w_{l,\widetilde K}:=\modulo{\det{\mathbb{J}_{\widetilde K}(\breve{\bm{b}}_l)}}\breve{w}_{l}\quad\mbox{and}\quad
  {\bm{b}}_{l,\widetilde K}:=T_{\widetilde K}(\breve{\bm{b}}_l).
\end{align}

\begin{definition}[Numeric sesquilinear and antilinear form]
\label{def:numsesqantcurv}
Let $Q_\rK^1$, $Q_\rK^2$ and $Q_\rK^3$ be three distinct quadrature rules
as in Definition \ref{def:quad}. For each $\widetilde{\tau_h}\in\{\widetilde\tau_{h_i}\}_{i\in\IN}$, we denote by
$\widetilde{\Phi}_{h}({\cdot},{\cdot})$ and $\widetilde\bF_{h}(\cdot)$
the perturbed, discrete, sesquilinear and antilinear forms over
fields in $\bm{P}^c(\widetilde\tau_{h})$, where exact integration is
replaced by numerical integration

\begin{align*}
&\widetilde{\Phi}_{h}({\bU_{h}},{\bV_{h}}):=\sum_{\widetilde K\in\widetilde\tau_{h}}
Q_{\widetilde K}^{1}(\mu^{-1}\curl\bU_{h}\cdot\curl\overline{\bV_{h}})
+Q_{\widetilde K}^{2}(-\omega^2\epsilon\bU_{h}\cdot\overline{\bV_{h}}),\\
&\widetilde{\bF}_{h}({\bV_{h}}):=\sum_{\widetilde K\in\tau_{h}}Q_{\widetilde K}^3(-\imath\omega\bJ\cdot \overline{\bV_{h}}),
\end{align*}
where, for $i=1,2,3$, $Q_{\widetilde K}^i$ is built from $Q_\rK^i$ as in \eqref{eq:Qcurv}.
\end{definition}

\begin{problem}[Discrete numerical problem on curved meshes]
\label{prob:varprobnumcurv}
Find $\widetilde{\bE}_h\in\bm P^c_0(\widetilde{\tau}_h)$ such that,
\begin{align*}
\widetilde\Phi_{h}(\widetilde{\bE}_h,{\bV}_h)=
\widetilde\bF_{h}({\bV}_h),
\end{align*}
for all ${\bV}_h\in\bm P^c_0(\tau_h)$.
\end{problem}

Let $\bE$, $\bE_h$ and $\widetilde\bE_h$ be the respective solutions of
Problems \ref{prob:varprob}, \ref{prob:varprobdiscurv} and
\ref{prob:varprobnumcurv}. Our current objective is to study the differences
between the convergence rates of $\bE_h$ and $\widetilde\bE_h$
to (an appropriate extension to the hold all $\D^H$ of) $\bE$. The following modification
of Strang's Lemma (Lemma \ref{lemma:Strang}) will prove useful. Its proof comes
from small variations of that in \cite[Thm.~4.2.11]{SauterSchwabBEM}, so we omit it
for brevity. We shall emulate \cite{di2018third} and consider a sequence of mappings
$\{\cJ_{h_i}\}_{i\in\IN}$ such that for all $i\in\IN$
\begin{align}\label{eq:mappingsequence}
\cJ_{h_i}:\hcurl{\D}\mapsto\hcurl{\D_{h_i}}.
\end{align}
These mappings need not be linear, but it is assumed
$\cJ_{h}\bE$ possesses some information on $\bE$, so
that the computation of
\begin{align}
\label{eq:meaningfulestimate}
\norm{\cJ_{h_i}(\bE)-\widetilde\bE_{h_i}}{\hcurl{\D_{h_i}}}
\end{align}
is in some way meaningful \cite[Rmk.~9]{di2018third}.
Indeed, the estimates in \cite{ciarlet1972combined}
may be interpreted for the specific choice of $\cJ_{h_i}\bE$
as an extension to $\D^H$ of $\bE$ for all $i\in\IN$.
Notice that in \cite{di2018third}, the authors assume
\begin{align*}
\cJ_{h_i}:\hcurl{\D}\mapsto\bm P^c_0(\widetilde\tau_{h_i}),
\end{align*}
which they require to estimate \eqref{eq:meaningfulestimate}.

\begin{lemma}(Modified Strang's lemma)
\label{lemma:StrangC}
Let $\Phi$ in \eqref{eq:Phi} satisfy Assumption \ref{ass:sesqform}
and \ref{ass:sesqformdis} and let $\bE$ and $\bE_{h_i}$ be the solutions
to Problems \ref{prob:varprob} and \ref{prob:varprobdiscurv}.
If the sequence of sesquiliear forms
$\{\widetilde\Phi_{h_i}\}_{i\in\IN}$ given by
Definition \ref{def:numsesqantcurv} satisfies, for all $\bU_{h_i},\;\bV_{h_i}\in\bm P^c_0(\widetilde\tau_{h_i})$,
\begin{align*}
\modulo{\Phi_{h_i}(\bU_{h_i},\bV_{h_i})-\widetilde\Phi_{h_i}(\bU_{h_i},\bV_{h_i})}
\leq c\; h_i^m\norm{\bU_{h_i}}{\hscurl{\D_{h_i}}{m}}\norm{\bV_{h_i}}{\hcurl{\D_{h_i}}},
\end{align*}
for $m\in\IN$ and a fixed and positive constant $c$ independent of the
mesh-size, then there is some $\ell\in\IN$ such that for all
the meshes in the sequence $\{\widetilde\tau_{h_{i}}\}_{i\in\IN,\ i>\ell}$
there exists a unique solution to Problem \ref{prob:varprobnumcurv},
$\widetilde\bE_{h_i}\in\bm P^c_0(\widetilde\tau_{h_i})$ and
\begin{alignat}{2}
\begin{aligned}
&\norm{\cJ_{h_i}(\bE)-\widetilde\bE_{h_i}}{\hcurl{\D_{h_i}}} &&\\
&\leq C_S\Bigg(
\norm{\cJ_{h_i}(\bE)-\bE_{h_i}}{\hcurl{\D_{h_i}}}+&&
\sup_{\bV_{h_i}\in\bm P^c_0(\tau_{h_i})\setminus\{\bnul\}}
\frac{\vert{\Phi_{h_i}(\bE_{h_i},\bV_{h_i})-\widetilde\Phi_{h_i}(\bE_{h_i},\bV_{h_i})}\vert}{\norm{\bV_{h_i}}{\hcurl{\D_{h_i}}}} \\
 & &&\qquad+\sup_{\bV_{h_i}\in\bm P^c_0(\tau_{h_i})\setminus\{\bnul\}}
\frac{\vert{\bF_{h_i}(\bV_{h_i})-\widetilde\bF_{{h_i}}(\bV_{h_i})}\vert}{\norm{\bV_{h_i}}{\hcurl{\D_{h_i}}}} \Bigg),
\end{aligned}\label{eq:curvstrangestimate}
\end{alignat}
for a fixed positive constant $C_S$, independent of the mesh-size,
and where the sequence of mappings $\{\cJ_{h_i}\}_{i\in\IN}$ is such that
\eqref{eq:mappingsequence} holds.
\end{lemma}

We will focus on providing estimates for the last two terms
in the right hand side of \eqref{eq:curvstrangestimate}. The
error induced by the approximation of $\D$ by $\D_{h_i}$
(i.e.~the first term in \eqref{eq:curvstrangestimate})
lies beyond the scope of this article.  
We continue by stating the main result of this section.
\begin{theorem}[Error estimate in affine meshes. Main result of Section \ref{sec:CurvDom}]
\label{thm:mainrescurv}
Let $\bE$ be the unique solution to Problem \ref{prob:varprob} and
suppose the following of the data of Problem \ref{prob:varprob}:
\begin{align*}
\bJ\in\bW^{m,q}(\D^H),\quad\mbox{and}\quad \epsilon_{i,j},\;(\mu^{-1})_{i,j}\in W^{m,\infty}(\D^H)\quad\forall\;i,\;j\in\{1,2,3\}
\end{align*}
for some positive $m\in\IN$ and $q\in\IR$ such that $$m>1,\quad q>2\quad\mbox{and}\quad q\geq\frac{m}{3}.$$
Then, if the quadrature rules used to build
$\{\widetilde\Phi_{h_i}\}_{i\in\IN}$ and $\{\widetilde\bF_{h_i}\}_{i\in\IN}$ in
Definition \ref{def:numsesqantcurv} are such that:
\begin{itemize}
\item $Q^1_{\rK}$ is exact for polynomials
of degree $k+\mathfrak{K}+m-3$,
\item $Q^2_{\rK}$ is exact for polynomials
of degree $k+2\mathfrak{K}+m-3$ and
\item $Q^3_{\rK}$ is exact for polynomials
of degree $k+2\mathfrak{K}+m-3$,
\end{itemize}
there exists some $\ell\in\IN$
such that for all $i>\ell$ there exists a unique solution
$\widetilde\bE_{h_i}\in\bm P^c_0(\tau_{h_i})$ to Problem
\ref{prob:varprobnum} and the solutions satisfy
\begin{align*}
&\norm{\cJ_{h_i}(\bE)-\widetilde\bE_{h_i}}{\hcurl{\D_{h_i}}} 
\\
&\leq C_S\Bigg(
\norm{\cJ_{h_i}(\bE)-\bE_{h_i}}{\hcurl{\D_{h_i}}}+
C_1h^m\norm{\bE_{h_i}}{\hscurl{\D_{h_i}}{m}}+C_2h^m \Bigg),
\end{align*}

where the positive constants $C_1$ and $C_2$ are independent of
the mesh-size, but depend on the parameters of Problem \ref{prob:varprob}
($\mu$, $\epsilon$, $\omega$, $\bJ$ and $\D$).
\end{theorem}

\subsection{Consistency error estimates and proof of Theorem \ref{thm:mainrescurv}}
As in Section \ref{sec:PolyDom}, we shall find error estimates for
the integrals over curved tetrahedrons $\widetilde K\in\widetilde\tau_h$.
The most notorious difference in the proofs for the following estimates
and those presented in the previous section are due to the fact that, if
$\bU_h\in\bm P^c_{\widetilde{K}}$ for some curved tetrahedron
$\widetilde K$, then $\bU_h\circ T_{\widetilde K}$ is, in general,
not a polynomial, so we can not apply the Bramble-Hilbert Lemma
as easily as before. We will see, however, that in our case
$\bU_h\circ T_{\widetilde K}\det\mathbb{J}_{\widetilde K}$ will
be a polynomial of a certain degree (higher than $k$) which will
allow us to proceed as before.

\begin{lemma}\label{lem:ConErrtildeK}
Let $\widetilde{K}\in\widetilde\tau_h$, $q\geq 1$,
$q'=\frac{q}{q-1}$, $m\in\IN$
with $m>\frac{3}{q}$ and $\bM(\bx)\in\IC^{3\times 3}$,
with $\bM=(M_{i,j})_{i,j=1}^3$,
such that $M_{i,j}\in W^{m,\infty}(\widetilde{K})$ for all $i$, $j\in\{1,2,3\}$.
If $Q_{\breve{K}}$ is a quadrature as in Definition \ref{def:quad}
which is exact for polynomials of degree ${k+2\mathfrak{K}+m-3}$
then, for all $\bU_h$, $\bV_h\in\bm{P}^c_{\widetilde{K}}$, the quadrature error
\begin{equation*}
\cE_{\widetilde K}(\bM\bU_h\cdot\bV_h):=\int_{\widetilde K}\bM\bU_h\cdot\bV_h\d\! \bx
-Q_{\widetilde{K}}\left(\bM\bU_h \cdot\bV_h\right),
\end{equation*}
is such that 
\begin{align*}
\modulo{\cE_{\widetilde{K}}(\bM\bU_h\cdot\bV_h)}\leq
CC_{\bM}h^{m}\norm{\bU_h}{\bm{W}^{m,q}(\widetilde{K})}\norm{\bV_h}{\Lp{q'}{\widetilde{K}}},
\end{align*}
where $$C_{\bM}:=\sum_{i,j=1}^{3}\norm{M_{i,j}}{W^{m,\infty}(\widetilde{K})},$$
and $C$ is a positive constant independent of $h$, $K$ and $\bM$.

\end{lemma}
\begin{proof}
Let $\bm{\phi}\in \bm{W}^{m,\infty}(\breve{K})$ and
$\bV_h\in \mathbb{P}_{k+2(\mathfrak{K}-1)}(\breve K;\IC^{3})$. Then,
\begin{align*}
  \modulo{\mathcal{E}_{\breve{K}}(\bm{\phi}\cdot\bV_h)}&\leq
  C_{\mathcal{E}}\norm{\bm{\phi}\cdot \bV_h}{L^\infty(\breve{K})}\\
  &\leq
  C_{\mathcal{E}}\norm{\bm{\phi}}{\bm{L}^\infty(\breve{K})}
  \norm{\bV_h}{\bm{L}^\infty(\breve{K})}\\
  &\leq
  C_{\mathcal{E}}\norm{\bm{\phi}}{\bm{W}^{m,q}(\breve{K})}
  \norm{\bV_h}{\Lp{q'}{\breve{K}}},
\end{align*}
for some positive $C_{\mathcal{E}}$---depending only on
$\breve{K}$---and $\frac{1}{q}+\frac{1}{q'}=1$. Also, since
the error is zero for all $\bm\phi\in \mathbb{P}_{m-1}(\breve{K};\IC)$
we have, by the Bramble-Hilbert Lemma (Lemma \ref{lemma:BHvec})
\begin{align}
  \modulo{\mathcal{E}_{\breve{K}}(\bm{\phi}\cdot\bV_h)}&\leq
  C_{\breve{K}}\seminorm{\bm{\phi}}{\bm{W}^{m,q}(\breve{K})}
  \norm{\bV_h}{\Lp{q'}{\breve{K}}}\quad\forall\;\bm{\phi}\in\bm{W}^{m,q}(\breve{K}),\label{eq:errKbcur}
\end{align}
for some positive constant $C_{\breve{K}}$ depending only on $\breve{K}$ and $C_{\cE}$.
Notice that
\begin{align*}
\int_{\widetilde K}\bM\bU_h\cdot\bV_h\d\! \bx=\int_{\breve K}(\bM\circ {T}_{\widetilde K})(\bU_h\circ T_{\widetilde K})\cdot(\bV_h\circ T_{\widetilde K})\det{\mathbb{J}_{\widetilde K}}\d\! \bx,\\
Q_{\widetilde{K}}\left(\bM\bU_h \cdot\bV_h\right)=Q_{\breve{K}}\left( (\bM\circ {T}_{\widetilde K})(\bU_h\circ T_{\widetilde K})\cdot(\bV_h\circ T_{\widetilde K})\det{\mathbb{J}_{\widetilde K}}\right).
\end{align*}
Hence, we need only find a bound for 
\begin{align*}
\cE_{\breve K}\left( (\bM\circ {T}_{\widetilde K})(\bU_h\circ {T}_{\widetilde K})\cdot(\bV_h\circ {T}_{\widetilde K})\det{\mathbb{J}_{\widetilde K}}\right).
\end{align*}
From the last equation we see that we require
$(\bV_h\circ {T}_{\widetilde K})\det{\mathbb{J}_{\widetilde K}}$
to be a polynomial. Indeed,
\begin{align}
\bV_h\circ {T}_{\widetilde K}\det{\mathbb{J}_{\widetilde{K}}}=\det{\mathbb{J}^{\top}_{\widetilde{K}}}\mathbb{J}_{\widetilde{K}}^{-\top}\left(\mathbb{J}_{\widetilde{K}}^{\top}(\bV_h\circ{T}_{\widetilde K})\right)=\text{Co}(\mathbb{J}_{\widetilde{K}})\left(\mathbb{J}_{\widetilde{K}}^{\top}(\bV_h\circ{T}_{\widetilde K})\right)\in\mathbb{P}_{k+2(\mathfrak{K}-1)},\label{eq:VhcircTpoly}
\end{align}
where $\text{Co}(\mathbb{J}_{\widetilde{K}})$ is the
pointwise cofactor matrix of $\mathbb{J}_{\widetilde{K}}$.
Then, from \eqref{eq:errKbcur} and \eqref{eq:VhcircTpoly} we get
\begin{align}
&\cE_{\breve K}\left( (\bM\circ {T}_{\widetilde K})(\bU_h\circ {T}_{\widetilde K})\cdot(\bV_h\circ {T}_{\widetilde K})\det{\mathbb{J}_{\widetilde K}}\right)
\nonumber\\
&\leq C_{\breve{K}}\norm{\det{{\mathbb{J}}_{\widetilde K}}}{\lp{\infty}{\breve{K}}}\seminorm{(\bM\bU_h)\circ T_{\widetilde K}}{W^{m,q}({\breve K})}\norm{\bV_h\circ T_{\widetilde K}}{\Lp{q'}{\breve K}}.\label{eq:Curvbound}
\end{align}
We continue by bounding each term in \eqref{eq:Curvbound},
beginning with the $\Lp{q'}{\breve{K}}$ norm of $\bV_h\circ T_{\widetilde{K}}$, which
is easily bounded through a change of variables
\begin{align*}
\norm{\bV_h\circ T_{\widetilde K}}{\Lp{q'}{\breve K}}\leq\norm{\bV_h}{\Lp{q'}{\widetilde K}}\frac{1}{\inf_{\bx\in\breve K}\modulo{\det{\mathbb{J}_{\widetilde K}}}^{\frac{1}{q'}}}.
\end{align*}
To bound $\seminorm{(\bM\bU_h)\circ T_{\widetilde K}}
{W^{m,q}({\breve K})}$, we proceed as in Lemma \ref{lem:ConErrK}
and employ the bound in Lemma \ref{lem:curvineqK},
\begin{align*}
\seminorm{(\bM\bU_h)\circ T_{\widetilde K}}{W^{m,q}({\breve K})}
&\leq c \sum_{i,j=1}^{3}\sum_{n=0}^{m}\seminorm{M_{i,j}\circ T_{\widetilde K}}{W^{m-n,\infty}(\breve{K})}
\seminorm{\bU_{h,j}\circ T_{\breve K}}{W^{n,q}(\breve{K})}\\
&\leq c \sum_{i,j=1}^{3}\sum_{n=0}^{m}h^{m-n}\seminorm{M_{i,j}}{W^{m-n,\infty}(\widetilde{K})}
\frac{1}{\inf_{\bx\in\breve K}\modulo{\det{\mathbb{J}_{\widetilde K}}}^{\frac{1}{q}}}h^{n}\seminorm{\bU_{h,j}}{W^{n,q}(\widetilde{K})},
\end{align*}
where the positive constant $c$ is independent of $\widetilde K$
and the mesh-size and may change from line to line.
Then, combining our computed bounds, \eqref{eq:Curvbound} and
Assumption \ref{ass:k_gregular}
\begin{align*}
\modulo{\cE_{\widetilde{K}}(\bM\bU_h\cdot\bV_h)}\leq
C\theta C_{\bM}h^{m}\norm{\bU_h}{\bm{W}^{m,q}(\widetilde{K})}\norm{\bV_h}{\Lp{q'}{\widetilde{K}}},
\end{align*}
for some positive $C$, independent of $\widetilde{K}$ and the mesh-size.
\end{proof}
Notice that if $m=1$ in Lemma \ref{lem:ConErrtildeK}, we are unable
to extract $\bm L^2$-norms of $\bU_h$ and $\bV_h$, hence our requirement
that $m>1$ in Theorem \ref{thm:mainrescurv}. Also, observe the differences
between the previous proof and that of Lemma \ref{lem:ConErrK}. Not only 
do we require stronger conditions from our quadrature rules, we are also
unable to bound $\seminorm{\bU_h\circ T_{\widetilde K}}{\bW^{n,q}(\rK)}$
as before, owing to the fact that $\bU_h\circ T_{\widetilde K}$ fails to be a
polynomial (this is also the reason why we require to introduce $q$ and $q'$).
\begin{lemma}\label{lem:ConErrtildeKCurl}
Let $\widetilde{K}\in\widetilde\tau_h$, $q\geq 1$,
$q'=\frac{q}{q-1}$, $m\in\IN$
with $m>\frac{3}{q}$ and $\bM(\bx)\in\IC^{3\times 3}$,
with $\bM=(M_{i,j})_{i,j=1}^3$,
such that $M_{i,j}\in W^{m,\infty}(\widetilde{K})$ for all $i$, $j\in\{1,2,3\}$.
If $Q_{\breve{K}}$ is a quadrature rule as in Definition \ref{def:quad}
such that it is exact for polynomials of degree ${k+\mathfrak{K}+m-3}$,
then, for all $\bU_h$, $\bV_h\in\bm{P}^c_{\widetilde{K}}$, the quadrature error
\begin{equation*}
\cE_{\widetilde K}(\bM\curl\bU_h\cdot\curl\bV_h):=\int_{\widetilde K}\bM\curl\bU_h\cdot\curl\bV_h\d\! \bx
-Q_{\widetilde{K}}\left(\bM\curl\bU_h \cdot\curl\bV_h\right),
\end{equation*}
with $Q_{\widetilde K}$ as in \ref{eq:Qcurv}, is such that 
\begin{align*}
\modulo{\cE_{\widetilde{K}}(\bM\curl\bU_h\cdot\curl\bV_h)}\leq
C\theta C_{\bM}h^{m}\norm{\curl\bU_h}{\bm{W}^{m,q}(\widetilde{K})}\norm{\curl\bV_h}{\Lp{q'}{\widetilde{K}}},
\end{align*}
where $$C_{\bM}:=\sum_{i,j=1}^{3}\norm{M_{i,j}}{W^{m,\infty}(\widetilde{K})},$$
and $C$ is a positive constant independent of $h$, $K$ and $\bM$.
\end{lemma}
\begin{proof}
We begin by noticing that for $\bV\in\hcurl{\widetilde{K}}$
\cite[Lem.~2.2]{jerez2017electromagnetic}
\begin{align*}
\curl \left(\mathbb{J}_{\widetilde{K}}^{\top}(\bV\circ T_{\widetilde{K}})\right)=
\det{\mathbb{J}_{\widetilde{K}}}\mathbb{J}_{\widetilde{K}}^{-1}(\curl\bV)\circ T_{\widetilde{K}},
\end{align*}
hence
\begin{align*}
\det{\mathbb{J}_{\widetilde{K}}}(\curl\bV)\circ T_{\widetilde{K}}=\mathbb{J}_{\widetilde{K}}\curl \left(\mathbb{J}_{\widetilde{K}}^{\top}(\bV\circ T_{\widetilde{K}})\right)\in \mathbb{P}_{k+\mathfrak{K}-2}(\breve{K};\IC).
\end{align*}
The proof proceeds as that for Lemma \ref{lem:ConErrtildeK}.
\end{proof}

\begin{lemma}\label{lem:ConErrKrhscurv}
Let $\widetilde K\in\widetilde \tau_h$, $m\in\IN$ and $q\in\IR$ such that
\begin{align}
q\geq 2\quad\mbox{and}\quad q>\frac{3}{m},\label{eq:q1condcurv}
\end{align}
and $Q_{\breve{K}}$ be a quadrature rule as in Definition \ref{def:quad}
such that it is exact on polynomials of degree ${k+2\mathfrak{K}+m-3}$.
Then, if $\bJ\in\bW^{m,q}(\widetilde K)$,
the local quadrature error $\cE_{\widetilde K}(\bJ\cdot\bV_h)$
(as defined in Lemma \ref{lem:ConErrtildeK}) is such that for all
$\bV_h\in\bm {P}^c_{\widetilde K}$
\begin{align*}
\modulo{\cE_{\widetilde K}(\bJ\cdot\bV_h)}\leq
Ch^{m}\modulo{\widetilde K}^{\frac{1}{2}-\frac{1}{q}}\norm{\bJ}{\bW^{m,q}(\widetilde K)}\norm{\bV_h}{0,\widetilde K},
\end{align*}
for a positive constant $C$ independent of $\widetilde K$, $\bJ$ and the mesh-size.
\end{lemma}
\begin{proof}
We pick up from an analogous expression as that in
\eqref{eq:errKbcur}, noticing that by similar arguments we get
\begin{align*}
  \modulo{\mathcal{E}_{\breve{K}}(\bm{\phi}\cdot\bV_h)}&\leq
  C_{\breve{K}}\seminorm{\bm{\phi}}{\bm{W}^{m,q}(\breve{K})}
  \norm{\bV_h}{0,{\breve{K}}}\quad\forall\;\bm{\phi}\in\bm{W}^{m,q}(\breve{K}),\;\bV_h\in\bm P^c_{\widetilde K}.
\end{align*}
Then, from our assumptions and Lemmas \ref{lem:detJmodK} and \ref{lem:curvineqK}
\begin{align*}
\cE_{\widetilde K}\left(\bJ\cdot\bV_h\right)&=\cE_{\rK}\left(\bJ\circ T_{\widetilde K}\cdot\bV_h\circ T_{\widetilde K}\det\mathbb{J}_{\widetilde K}\right)
\nonumber\\
&\leq C_{\breve{K}}\norm{\det{{\mathbb{J}}_{\widetilde K}}}{\lp{\infty}{\breve{K}}}\seminorm{\bJ\circ T_{\widetilde K}}{W^{m,q}({\breve K})}\norm{\bV_h\circ T_{\widetilde K}}{0,{\breve K}}\\
&\leq c\norm{\det{{\mathbb{J}}_{\widetilde K}}}{\lp{\infty}{\breve{K}}}h^m\frac{1}{\inf_{\bx\in\breve K}\modulo{\det{\mathbb{J}_{\widetilde K}}}^{\frac{1}{q}}}\norm{\bJ}{W^{m,q}({\widetilde K})}\frac{1}{\inf_{\bx\in\breve K}\modulo{\det{\mathbb{J}_{\widetilde K}}}^{\frac{1}{2}}}\norm{\bV_h}{0,{\widetilde K}}\\
&\leq c\;\theta{\modulo{\widetilde K}^{\frac{1}{2}-\frac{1}{q}}}{\modulo{\rK}^{\frac{1}{q}-\frac{1}{2}}}\;h^m\norm{\bJ}{W^{m,q}({\widetilde K})}\norm{\bV_h}{0,{\widetilde K}},
\end{align*}
for a positive constant $c$ independent of the mesh-size and $K$,
which may change from line to line.
\end{proof}

As before, the computed estimates will enable us to
prove, based in our assumptions, consistency error estimates
for the respective sesquilinear and antilinear forms considered
in this section. The proofs of the following two Theorems
(yielding said consistency estimates) are analogous to the proofs
of Theorems \ref{thm:consesq} and \ref{thm:conant} and
are thus omitted.

\begin{theorem}[Consistency error for the sesquilinear forms
$\{\widetilde\Phi_{h_i}\}_{i\in\IN}$]
\label{thm:consesqcurv}
Recall $k\in\IN$ as the polynomial degree of our approximation
spaces. Let $m\in\IN$ and assume the following of the quadrature
rules defining the sesquilinear forms in
$\{\widetilde\Phi_{h_i}\}_{i\in\IN}$ in Definition \ref{def:numsesqantcurv}:
\begin{itemize}
\item The quadrature rule $Q^1_{\rK}$ is exact for polynomials
of degree $k+\mathfrak{K}+m-3$.
\item The quadrature rule $Q^2_{\rK}$ is exact for polynomials
of degree $k+2\mathfrak{K}+m-3$.
\end{itemize}
Then, under Assumptions \ref{ass:polydom} and \ref{ass:meshPoly} and if the
coefficients of $\mu^{-1}$ and $\epsilon$ belong to
$W^{m,\infty}(\D^H)$ and for $\Phi_{h_i}$ as in \eqref{eq:Phih},
\begin{align*}
&\modulo{\Phi_{h_i}(\bU_{h_i},\bV_{h_i})-\widetilde{\Phi}_{h_i}(\bU_{h_i},\bV_{h_i})}\\
&\leq C_{\Phi}h_i^{m} \sum_{\widetilde K\in\tau_{h_i}}C_{\mu^{-1}}\norm{\curl\bU_{h_i}}{\bm{W}^{m,q}(\widetilde{K})}\norm{\curl\bV_{h_i}}{\Lp{q'}{\widetilde{K}}}
+\omega^2C_{\epsilon}\norm{\bU_{h_i}}{\bm{W}^{m,q}(\widetilde{K})}\norm{\bV_{h_i}}{\Lp{q'}{\widetilde{K}}}
\end{align*}
for all $\bU_{h_i}$, $\bV_{h_i}\in\bm{P}_0^c(\tau_{h_i})$, where $q$ and $q'\in\IR$
are such that $q>1$, $q'=\frac{q}{q-1}$ and $q>\frac{3}{m}$, $C_{\mu^{-1}}$ and
$C_{\epsilon}$ are positive constants depending on $\mu^{-1}$ and
$\epsilon$, and $C_{\Phi}$ is a positive constant independent of the mesh sizes $\{h_i\}_{i\in\IN}$.
\end{theorem}
  \begin{theorem}[Consistency error for the antilinear forms
  $\{\widetilde{\bm F}_{h_i}\}_{i\in\IN}$]
\label{thm:conantcurv}
Recall $k\in\IN$ as the polynomial degree of our approximation
spaces. Let $m\in\IN$ and assume the quadrature rule $Q^3_{\breve{K}}$
from Definition \ref{def:numsesqant} is exact for polynomials of degree
${k+2\mathfrak{K}+m-3}$. Then, under Assumptions
\ref{ass:polydom} and \ref{ass:meshPoly} and if $\bJ$ is such that
$\bJ\in \bW^{m,q}(D)$ for $q > \frac{3}{m}$ and $q\geq 2$, then
\begin{align*}
\modulo{\bF_{h_i}(\bV_{h_i})-\widetilde{\bF}_{h_i}(\bV_{h_i})}\leq &
C_{\bF}\omega {h_i}^{m}\modulo{\D_{h_i}}^{\frac{1}{2}-\frac{1}{q}}
\norm{\bJ}{\bW^{m,q}(\D_{h_i})}\norm{\bV_{h_i}}{0,{\D_{h_i}}}
\end{align*}
for all $\bV_{h_i}\in\bm{P}_0^c(\tau_{h_i})$, for a positive constant $C_{\bF}$ independent of the mesh-size.
\end{theorem}
\begin{proof}[Proof of Theorem \ref{thm:mainrescurv}]
Our main result follows by combining Theorems
\ref{thm:consesqcurv} and \ref{thm:conantcurv} with
Lemma \ref{lemma:StrangC}.
\end{proof}

\section{Numerical Examples}
\label{sec:Num_Ex}
We test our main results on a number of simple numerical examples. In order to isolate the effects that quadrature rules have on the observed rates of convergence in Theorems \ref{thm:mainresaffine} and \ref{thm:mainrescurv}, we focus only on the case of a polygonal domain, namely the cube $\D:=[-1,1]^3\subset\IR^3$, and consider a coercive
problem with the following sesquilinear and antilinear forms:

\begin{align}
\label{eq:Phi_ex_1}
\Phi(\bU,\bV)&:=
\int_{\D} {(\mu_0\mathsf{I})}^{-1} \curl \mathbf{U} \cdot \curl \overline{\mathbf{V}}-
\omega^2(\epsilon_0\mathsf{I}) \mathbf{U} \cdot \overline{\mathbf{V}} \d\!\bx,\\
\label{eq:rhs_ex_1}
\bF(\bV)&:=-\imath\omega\int_{\D}\bJ\cdot\overline{\bV}\d\!\bx,
\end{align}
where $\mathsf{I}\in\IR^{3\times3}$ is the identity matrix, $\omega=1$, $\mu_0$ and $\epsilon_0$ are strictly positive and negative scalars, respetively, $\bJ$ is given by
\begin{align}
\bJ(\bx)=\frac{\imath}{\omega}\mu_0^{-1}(4-2x_2^2-x_3^2)[1,0,0]^\top-\omega\epsilon_0(x_2^2-1)(x_3^2-1)[1,0,0]^\top\in\bW^{1,q}(\D)\quad\forall q>2,
\label{eq:J_ex_1}
\end{align}
and the solution $\bE$ to Problem \ref{prob:varprob} is given by
\begin{align}
\bE(\bx)=(x_2^2-1)(x_3^2-1)[1,0,0]^{\top}\in\hocurl{\D}\cap\hscurl{\D}{2}.\label{eq:analyticSol}
\end{align}

Experiments were carried out using GETDP \cite{getdp} (version 3.3.0) and modifications to its source code corresponding to unimplemented
quadrature rules. The points and weights of the quadrature rules available in GETDP may be examined in the file \texttt{Gauss\_Tetrahedron.h} of the source code. We also used GMSH \cite{geuzaine2009gmsh} to generate the required meshes of the domain $\D$.
\subsection{Convergence estimate for first order finite elements ($k=1$)}
\label{ssec:num_1}
We consider first order curl-conforming finite elements and set
the parameters in \eqref{eq:Phi_ex_1}, \eqref{eq:rhs_ex_1} and \eqref{eq:J_ex_1} to $\mu_0=10$ and $\epsilon_0=-10$.
We construct two numerical variations for both the sesquilinear and antilinear forms, given by
\begin{align*}
\widetilde{\Phi}_{h,0}({\bU_h},{\bV_h})&:=\sum_{K\in\tau_h}
Q_{K}^{1}({(\mu_0\mathsf{I})}^{-1}\curl\bU_h\cdot\curl\overline{\bV_h})
+Q_{K}^{1}(-\omega^2{(\epsilon_0\mathsf{I})}\bU_h\cdot\overline{\bV_h}),\\
\widetilde{\Phi}_{h,1}({\bU_h},{\bV_h})&:=\sum_{K\in\tau_h}
Q_{K}^{1}({(\mu_0\mathsf{I})}^{-1}\curl\bU_h\cdot\curl\overline{\bV_h})
+Q_{K}^{2}(-\omega^2{(\epsilon_0\mathsf{I})}\bU_h\cdot\overline{\bV_h}),\\
\widetilde{\bF}_{h,0}({\bV_h})&:=\sum_{K\in\tau_h}Q_{K}^3(-\imath\omega\bJ\cdot\overline{\bV_h}),\\
\widetilde{\bF}_{h,1}({\bV_h})&:=\sum_{K\in\tau_h}Q_{K}^2(-\imath\omega\bJ\cdot\overline{\bV_h}),
\end{align*}
where $Q_{\breve K}^{1}$ and $Q_{\breve K}^{3}$ are one-point quadrature rules with arbitrarily different chosen quadrature points in $\breve K$ --exact on polynomials of degree zero-- and $Q_{\breve K}^{2}$ is a one-point Gaussian quadrature rule --exact on polynomials of degree one--. Quadratures $Q^1_K$, $Q^2_K$ and $Q^3_K$, for $K\in\tau_h$, are then built from $Q^1_{\breve K}$, $Q^2_{\breve K}$ and $Q^3_{\breve K}$ as
indicated in \eqref{eq:QK}. Hence, $\widetilde{\Phi}_{h,1}$
and $\widetilde{\bF}_{h,1}$ satisfy the requirements of Theorem \ref{thm:mainresaffine} (with $r=1$) while $\widetilde{\Phi}_{h,0}$
and $\widetilde{\bF}_{h,0}$ do not. Figure \ref{fig:error_test} displays the convergence in the $\hcurl{\D}$-norm of the solution to Problem \ref{prob:varprobnum} corresponding to the different numerical implementations of the sesquilinear and antilinear forms. Nine meshes with 168, 228, 1,242, 2,810, 9,188, 38,782, 119,134, 500,300 and 1,265,246 degrees of freedom were employed.

\begin{figure}[t]
\center
\includegraphics[scale=0.6]{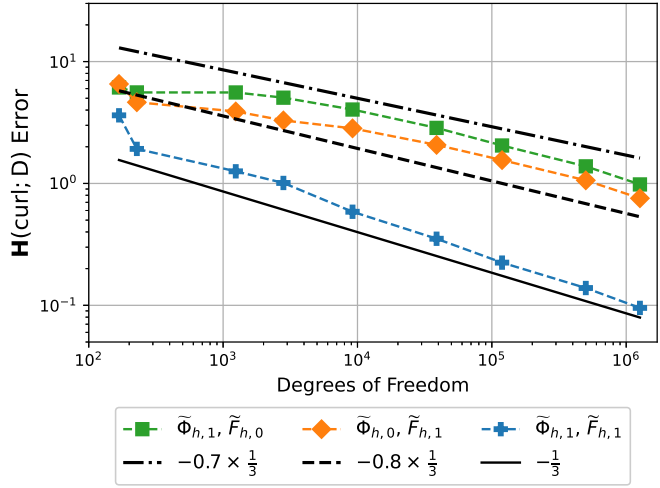}
\caption{Error convergence in the $\hcurl{\D}$--norm for the solutions of
Problem \ref{prob:varprobnum} to $\bE$ --given in \eqref{eq:analyticSol}-- depending on the implemented numerical variations of the sesquilinear and antilinear forms, indicated by the legend. The implementation considering $\widetilde\Phi_{h,\text{e}}$ and $\widetilde F_{h,\text{e}}$ attains a rate of convergence of $-\frac{1}{3}$ with respect to the number of degrees of freedom, as predicted by Theorem \ref{thm:mainresaffine}. On the other hand, implementations with at least one term not satisfying the conditions of Theorem \ref{thm:mainresaffine} suffer from a degenerated rate between $-0.7\times\frac{1}{3}$ and $-0.8\times\frac{1}{3}$, but convergence is still observed. }
\label{fig:error_test}
\end{figure}

\subsection{Convergence estimate for second order finite elements ($k=2$)}
\label{ssec:num_2}
We extend our previous experiment to second order curl-conforming finite elements. Again, we consider $\mu_0=10$ and $\epsilon_0=-10$, and construct three numerical variations for the sesquilinear form:
\begin{align*}
\widetilde{\Phi}_{h,0}({\bU_h},{\bV_h})&:=\sum_{K\in\tau_h}
Q_{K}^{1}({(\mu_0\mathsf{I})}^{-1}\curl\bU_h\cdot\curl\overline{\bV_h})
+Q_{K}^{1}(-\omega^2{(\epsilon_0\mathsf{I})}\bU_h\cdot\overline{\bV_h}),\\
\widetilde{\Phi}_{h,1}({\bU_h},{\bV_h})&:=\sum_{K\in\tau_h}
Q_{K}^{1}({(\mu_0\mathsf{I})}^{-1}\curl\bU_h\cdot\curl\overline{\bV_h})
+Q_{K}^{2}(-\omega^2{(\epsilon_0\mathsf{I})}\bU_h\cdot\overline{\bV_h}),\\
\widetilde{\Phi}_{h,2}({\bU_h},{\bV_h})&:=\sum_{K\in\tau_h}
Q_{K}^{2}({(\mu_0\mathsf{I})}^{-1}\curl\bU_h\cdot\curl\overline{\bV_h})
+Q_{K}^{2}(-\omega^2{(\epsilon_0\mathsf{I})}\bU_h\cdot\overline{\bV_h}),
\end{align*}
where $Q_{\breve K}^{1}$ is a $2\times 2\times 2$ tensorized  Gauss-Legendre quadrature rule --exact on polynomials of degree two on $\breve K$-- and $Q_{\breve K}^{2}$ is a five point Gaussian quadrature rule --exact on polynomials of degree three--. Hence, $\widetilde{\Phi}_{h,1}$
and $\widetilde{\Phi}_{h,2}$ satisfy the requirements of Theorem \ref{thm:mainresaffine} with $r=1$ and $2$, respectively. The right-hand side is implemented with a 15 point Gaussian quadrature and is left undisturbed throughout the experiments in this section. Figure \ref{fig:error_test_2} displays the convergence of the solution to Problem \ref{prob:varprobnum} corresponding to the different numerical implementations of the sesquilinear form. 8 meshes with 1,184, 1,688, 7,936, 1,7492, 54,480, 223,652, 674,676 and 2,454,312 degrees of freedom were employed.

\begin{figure}[t]
\center
\includegraphics[scale=0.60]{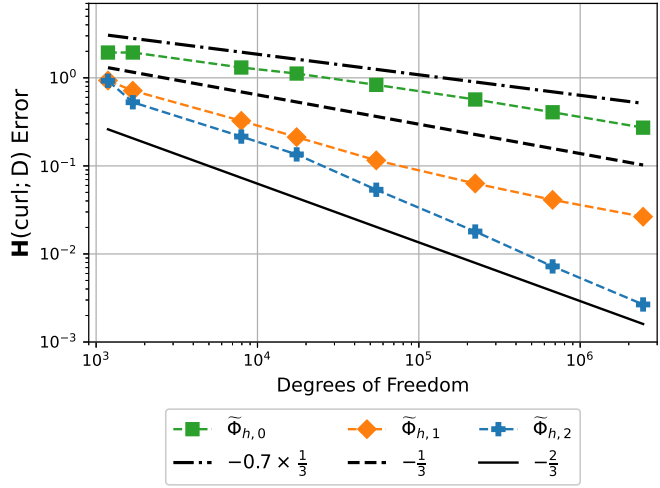}
\caption{Error convergence in the $\hcurl{\D}$-norm for the solutions of
Problem \ref{prob:varprobnum} to $\bE$ --given in \eqref{eq:analyticSol}-- depending on the implemented numerical variations of the sesquilinear and antilinear forms, indicated by the legend. The implementation considering $\widetilde\Phi_{h,2}$ and
$\widetilde\Phi_{h,1}$ attain their predicted rates of convergence of $-\frac{2}{3}$ and $-\frac{1}{3}$, respectively. The implementation considering $\widetilde\Phi_{h,0}$ suffers from a degenerated rate close to $-0.7\times\frac{1}{3}$, though convergence is still observed.}
\label{fig:error_test_2}
\end{figure}

\begin{remark}
\label{rmk:webb}
GETDP does not include an implementation of the second order
curl-conforming finite elements defined in Section \ref{ssec:FEMPoly}. However, an implementation of the second order Webb basis functions \cite{geuzaine1999convergence,webb1999hierarchal,webb1993hierarchal} is available. Since they are contained in $\mathbb{P}_2(\breve{K})$, the consistency estimates in Theorems \ref{thm:consesq} and \ref{thm:conant} remain valid, so our
numerical examples are still meaningful when using this alternative basis.
\end{remark}

\subsection{Effect of quadrature precision on preasymptotic convergence}
We investigate the effect that quadrature precision has on $\ell\in\IN$ in Theorem \ref{thm:mainresaffine}, i.e., the duration of the preasymptotic regime before convergence is observed at the 
predicted rate. We consider $\mu_0=10$ (as before)
and $\epsilon_0=-10-9\sin(m\pi x_3)$ for two cases $m=10$ and $20$. The solution to Problem \ref{prob:varprob} is still given by \eqref{eq:analyticSol}. We construct our numerical variations for the sesquilinear form as follows
\begin{align*}
\widetilde{\Phi}_{h,n}({\bU_h},{\bV_h})&:=\sum_{K\in\tau_h}
Q_{K}^1({(\mu_0\mathsf{I})}^{-1}\curl\bU_h\cdot\curl\overline{\bV_h})+Q_{K,n}^{2}(-\omega^2{(\epsilon_0\mathsf{I})}\bU_h\cdot\overline{\bV_h}),
\end{align*}
where $Q_{\breve K}^1$ is as before and $Q_{\breve{K},n}^2$ is a Gaussian
quadrature over $\breve{K}$ with $n=1$, $4$ and $15$ points. The right-hand side
is implemented with a 29 point Gaussian quadrature. Figure \ref{fig:l_precision} displays the convergence of the solution of
Problem \ref{prob:varprobnum} depending on the number of quadrature points used in the implementation of the sesquilinear
form. The employed meshes were as in Section \ref{ssec:num_1}.

\begin{figure}[t]
     \subfloat[$\epsilon_0=-10-9\sin(10\pi x_3)$\label{subfig:l_precision_10}]{%
       \includegraphics[width=0.5\textwidth]{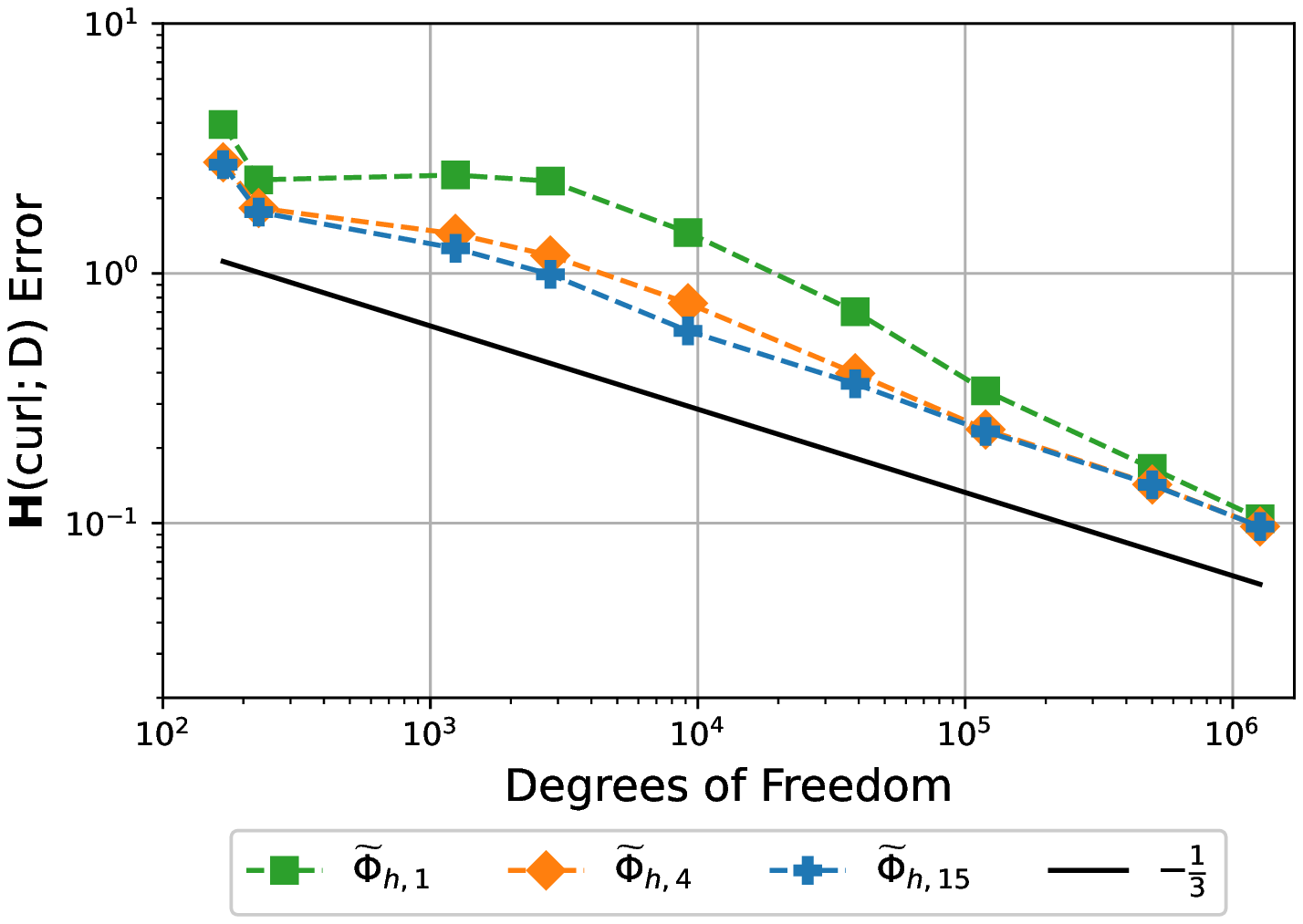}
     }
     \hfill
     \subfloat[$\epsilon_0=-10-9\sin(20\pi x_3)$\label{subfig:l_precision_20}]{%
       \includegraphics[width=0.5\textwidth]{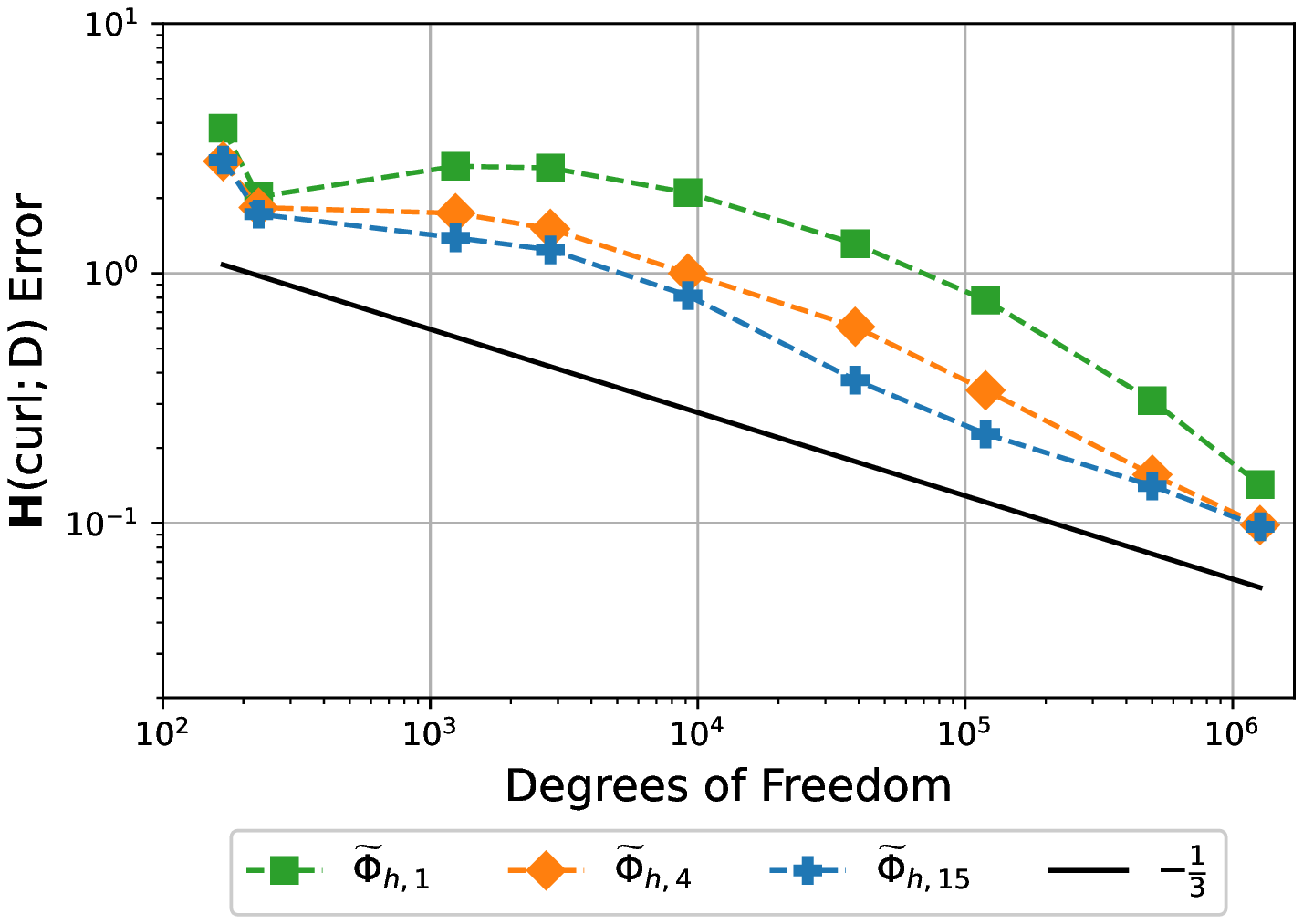}
     }
     \caption{Error convergence in the $\hcurl{\D}$-norm for solutions of
Problem \ref{prob:varprobnum} to $\bE$ --given in \eqref{eq:analyticSol}-- for the two cases of $\epsilon_0$ and depending on the implemented numerical variations of the sesquilinear form. In both cases, we observe a marked preasymptotic regime when only one quadrature point is implemented, where no convergence is observed before reaching $\approx10^4$ degrees of freedom
and the mesh is able to resolve the oscillatory term in $\epsilon_0$. Improving the quadrature rule to 4 points quickly corrects this issue on the case in Figure \ref{subfig:l_precision_10}, though a preasymptotic regime is still observed on the case displayed in Figure \ref{subfig:l_precision_20}, which is improved on by a further increase in precision to 15 quadrature points.}
     \label{fig:l_precision}
   \end{figure}

\section{Concluding Remarks}
\label{sec:Conc}
Our two main results (Theorems \ref{thm:mainresaffine}
and \ref{thm:mainrescurv}) yield sufficient conditions to
ensure convergence rates for the errors induced by quadrature
rules used when solving Maxwell Equations via the FE method with
inhomogeneous coefficients and on meshes with curved elements
(tetrahedrons). Interestingly, Theorem \ref{thm:mainresaffine}
confirm the presumptions of P.~Monk in the penultimate paragraph of
Section 8.3 in \cite{Monk:2003aa}, where it is stated that quadrature
rules exact on polynomials of degree $2k-1$ are expected to yield
convergence rates of order $h^k$. 

Unlike our result in Section \ref{sec:PolyDom}, Theorem \ref{thm:mainrescurv}
analyses only the quadrature effect in implementation and
does not present convergence estimates for the fully discrete solution to the real
solution. The result does, however, set aside the issue of numerical integration, so
that only the variational crime of the approximation of the real domain is left to
be analysed. Notice as well, that choosing $\mathfrak{K}=1$ in Section \ref{sec:CurvDom}
yields the same conditions for the quadrature rules in our two main results, which is of course
to be expected.

Numerical examples in Section \ref{sec:Num_Ex} not only confirm our results, but also display
the necessity of the conditions of Theorems \ref{thm:mainresaffine} and \ref{thm:mainrescurv}, since implementations that do not satisfy said conditions attain lower convergence 
rates than implementations that do.

Lastly, though we consider the smoothness of our parameters $\epsilon$, $\mu$ and $\bJ$ to
be global---belonging to some Sobolev space on the whole domain---one could quite easily accommodate
our results to consider parameters with piecewise smoothness on a finite set of sub-domains of
$\D$, by requesting that the two dimensional surfaces across which they fail to posses the required
degree of smoothness, does not cross any element of the mesh. In other words, that each of the sub-domains
is meshed so that the parameters are smooth on all elements of the mesh. An analogous consideration
holds if the domain $\D$ fails to posses a certain degree of smoothness at a finite number of points on its
boundary.



\bibliographystyle{plain}
\bibliography{max}
\end{document}